\documentclass[11pt]{article}

\usepackage[style=ieee, hyperref=true, backref=true, backend=biber]{biblatex}
\usepackage{xpatch}

\usepackage{graphicx}
\usepackage[margin=25mm]{geometry}
\usepackage[font=sf]{caption}

\usepackage{mathrsfs}
\usepackage{amsmath}
\usepackage{bm}
\usepackage{amsfonts}
\usepackage{amssymb}
\usepackage{siunitx}
\usepackage{verbatim}
\usepackage[hidelinks]{hyperref} 

\usepackage{amsthm}
\usepackage{romannum}
\usepackage{mathtools}
\usepackage{cleveref}
\usepackage{multicol}
\usepackage{multirow}
\usepackage{booktabs}
\usepackage{acronym}
\usepackage{subfigure}
\AtBeginDocument{\pagenumbering{arabic}}

\usepackage{algorithm}
\usepackage[algo2e,linesnumbered, ruled]{algorithm2e}

\usepackage{caption}
\newtheorem{remark}{Remark}
\newtheorem{corollary}{Corollary}
\newtheorem{theorem}{Theorem}
\newtheorem{lemma}{Lemma}
\newtheorem{assumption}{Assumption}
\newtheorem{proposition}{Proposition}
\newtheorem{inverse_problem}{Inverse Problem}

\numberwithin{definition}{section}
\numberwithin{remark}{section}
\numberwithin{corollary}{section}
\numberwithin{theorem}{section}
\numberwithin{lemma}{section}
\numberwithin{assumption}{section}
\numberwithin{proposition}{section}
\numberwithin{inverse_problem}{section}
\numberwithin{forward_problem}{section}

\newcommand{\norm}[1]{\left\lVert #1 \right\rVert}

\newcommand{\R}{\mathbb{R}}
\newcommand{\N}{\mathbb{N}}

\newcommand{\cR}{\mathcal{R}}

\newcommand{\cJ}{\mathcal{J}}

\newcommand{\divergence}{\nabla \cdot}
\newcommand{\symgrad}{\hat{\nabla}}
\newcommand{\trace}{\operatorname{tr}}

\newcommand{\defeq}{\overset{\text{\tiny def}}{=}}

\newcommand\restr[2]{{
  \left.\kern-\nulldelimiterspace 
  #1 
  \vphantom{\big|} 
  \right|_{#2} 
  }}

\newcommand{\bu}{\mathbf{u}}

\newcommand{\bn}{\mathbf{n}}

\newcommand{\p}{\partial}

\newcommand{\vect}{\mathbf}
\newcommand{\mvect}{\boldsymbol}
\newcommand{\mat}{\mathbf}
\newcommand{\tens}{\mathbf}

\usepackage{authblk}
\title{On Parameter Identification in Three-Dimensional Elasticity and Discretisation with Physics-Informed Neural Networks}

\author[1,2,3]{Federica Caforio}
\author[3,4]{Martin Holler}
\author[3,4]{Matthias Höfler}

\affil[1]{Department of Mathematics and Scientific Computing, University of Graz, NAWI Graz, Austria}
\affil[2]{Gottfried Schatz Research Center: Division of Biophysics, Medical University of Graz, Austria}
\affil[3]{BioTechMed-Graz, Austria}
\affil[4]{IDea\_Lab, University of Graz, Austria}

\date{}

\begin{document}

\maketitle

\begin{center}
\vspace{-1cm}
  Emails follow the format: \texttt{firstname.lastname@uni-graz.at}
\end{center}

\begin{abstract}
Physics-informed neural networks have emerged as a powerful tool in the scientific machine learning community, with applications to both forward and inverse problems. While they have shown considerable empirical success, significant challenges remain—particularly regarding training stability and the lack of rigorous theoretical guarantees, especially when compared to classical mesh-based methods.

In this work, we focus on the inverse problem of identifying a spatially varying parameter in a constitutive model of three-dimensional elasticity, using measurements of the system’s state. This setting is especially relevant for non-invasive diagnosis in cardiac biomechanics, where one must also carefully account for the type of boundary data available.

To address this inverse problem, we adopt an all-at-once optimisation framework, simultaneously estimating the state and parameter through a least-squares loss that encodes both available data and the governing physics. For this formulation, we prove stability estimates ensuring that our approach yields a stable approximation of the underlying ground-truth parameter of the physical system independent of a specific discretisation. We then proceed with a neural network-based discretisation and compare it to traditional mesh-based approaches. Our theoretical findings are complemented by illustrative numerical examples.

\end{abstract}

\section{Introduction}
Physics-informed neural networks (PINNs) have recently attracted significant attention within the scientific computing and machine learning communities and nowadays represent a relevant tool for the numerical solution of partial differential equations (PDEs). 
Despite their empirical success, as with many nonlinear parametrisations common in machine learning, rigorous theoretical guarantees regarding their convergence and generalisation remain limited.

Early contributions to the approximation theory of PINNs for forward problems include the works of \cite{de_ryck_error_2022} and \cite{mishra_estimates_2022}. 
These studies provide generalisation error bounds that can be reinterpreted as best-approximation errors with respect to the physical state.
This perspective is natural, given that the actual optimisation error is both difficult to quantify and inherently stochastic. 
A subsequent study by the same authors \cite{mishra_estimates_2023} extends these considerations to inverse problems, leveraging conditional stability estimates to derive generalisation bounds. The analysis in \cite{mishra_estimates_2023} mainly addresses unique continuation problems, which benefit from a rich literature on conditional stability (e.g., \cite{alessandrini_stability_2009}).

Nonetheless, many inverse problems of practical interest — particularly parameter identification — do not fall into the class of unique continuation problems. Instead, these involve estimating spatially varying parameters in PDE models. 
In principle, provided that suitable stability estimates exist, the same techniques can be applied to derive generalisation bounds.
However, it remains unclear how this translates not only into generalisation bounds of the state but also of the sought parameters.

Recent applications of PINNs to parameter identification have shown promising performance, not only in forward simulations but also in inverse formulations. 
In the context of forward problems, we refer, for instance, to \cite{wang_simulating_2025, bai_physics-informed_2023, doumeche_convergence_2023}. 
Applications to inverse problems in elasticity include \cite{kamali_elasticity_2023, haghighat_physics-informed_2021}, with particular emphasis on cardiac modelling, where PINNs have been employed to infer local mechanical properties of myocardial tissue \cite{caforio_physics-informed_2024, hofler_physics-informed_2025}. 
These advances motivate a systematic theoretical study of parameter identification in three-dimensional elasticity within the PINN framework.

To enable such analysis, we follow the theoretical pathway established by \cite{mishra_estimates_2023}, which requires a well-posed physical system admitting conditional stability estimates.
Consequently, we restrict our attention to certain classes of systems arising in linear, quasi-static, three-dimensional elasticity. 
Already in this setting of linear, quasi-static, three-dimensional elasticity, one cannot expect to have such stability estimates available for the general case \cite{isakov_weak_2010, lin_detecting_2018}. 
However, building on the stability results developed in \cite{lin_detecting_2018, di_fazio_stable_2017}, we aim at establishing theoretical properties for the inverse problem of identifying spatially varying parameters for these classes of constitutive laws.

In this work, we:
\begin{enumerate}
	\item 
    Extend the conditional stability estimates from \cite{lin_detecting_2018, di_fazio_stable_2017} to a form that is consistent with the theoretical framework presented in these works and incorporates practical considerations for identifying local material properties in in-vivo myocardial tissue — where direct access to the parameter (e.g., through pointwise or boundary measurements) is typically not available, and only indirect data such as boundary traction measurements can be used.
	\item Derive approximation-independent properties of the all-at-once optimisation problem, where the state and the parameter are jointly inferred.
	\item Establish generalisation (or discretisation) error estimates for the parameter of interest, following the approach of \cite{mishra_estimates_2022}.
	\item Provide a numerical comparison between the PINN-based method and finite element (FE)-based approaches, both in an all-at-once as well as in a reduced setting.
\end{enumerate}

\subsection{Basic notation}
Let $\Omega \subset \R^3$ denote the physical domain of interest, representing a bounded region with boundary $\p \Omega$. 
The inner part of $\Omega$ with a distance of at least $d>0$ to the boundary is denoted by $\Omega_d=\{x \in \Omega: \operatorname{dist}(x, \partial \Omega)>d\}$. 
We adopt the following conventions for notational clarity throughout this work: Scalars are denoted by regular lowercase letters, e.g., $t \in \R$, vectors are represented by bold lowercase letters, e.g., $\vect u \in \R^3$, matrices and tensors are written in bold uppercase letters, e.g., the identity matrix $\mat I \in \R^{3 \times 3}$. 
For a sufficiently smooth vector field $\vect u: \R^3 \longrightarrow \R^3$, we denote the gradient by $\nabla \vect u$, the symmetric gradient with $\symgrad \vect u \defeq \frac{1}{2}\left( \nabla \vect u + \nabla \vect u ^T \right)$, and its divergence with $\divergence \vect u$.
Similarly, for a sufficiently smooth tensor field $\tens U: \R^3 \longrightarrow \R^{3 \times 3}$, we denote the tensor divergence also by $\divergence \tens U$, which is applied row-wise.
Scalar products are denoted by $\vect u \cdot \vect v$, the tensor product by $\vect u \otimes \vect v$, and the Frobenius product by $\mat U : \mat V$.
We denote by $\delta_{ij}$ the delta function, which is 1 for the case $i=j$ and 0 otherwise.

The space of Lebesgue-integrable scalar functions where the $p$-th power is integrable is denoted by $L^p(\Omega, \R^3)$ for vector-valued functions and by $L^p(\Omega)$ for scalar functions. 
The space of functions whose $k$-th derivative is Lipschitz continuous is denoted by $\mathcal C^k(\Omega)$, whereas $\mathcal C^k(\bar \Omega)$ consists of functions $f \in \mathcal C^k(\Omega)$ whose derivatives can be continuously extended to $\bar \Omega$ (see e.g. \cite[Section 3.6]{alt_linear_2016}).
The Sobolev space of vector-valued functions with weak derivatives up to order $k$ in $L^p$ is denoted by $W^{k,p}(\Omega, \R^3)$, and for the special case of a Hilbert space structure by $H^k \defeq W^{k,2}$.

Strong convergence is denoted by an arrow, e.g., $\vect u_n \longrightarrow \vect u$, weak convergence by $\vect u_n \rightharpoonup \vect u$, and weak* convergence by $\vect u_n \overset{*}{\rightharpoonup} \vect u$.

We denote by $f(x)=\mathcal O(g(x))$ an asymptotic upper bound (in practice basically up to constant factors, i.e. $\norm{f} \leq C \norm{g}$ for suitably chosen norms) and by $g(x)=\Omega(f(x))$ an asymptotic lower bound (which can also be regarded as a bound up to constant factors).

\section{Stability}

We begin by presenting the conceptual framework adopted in this work, which comprises the governing equations of the physical system and the general strategy for addressing the associated inverse problem. Central to our approach is the formulation of an objective functional subject to minimisation. The central insight is that a stability property of the physical system can be systematically transferred to ensure stability within the corresponding optimisation formulation.

Throughout this article, we assume the following regularity on the domain.
\begin{assumption}[Regularity of the domain]\label{met:ass:domain}
    We assume that $\Omega \subset \R^3$ is a bounded domain with $\mathcal{C}^{1}$ boundary.
\end{assumption}

\subsection{Linear elasticity}
To model three-dimensional linear elasticity, we define the governing PDE system involving the ground-truth displacement field $\vect u^* \in H^1(\Omega, \R^3)$ that satisfies the overdetermined boundary value problem
\begin{equation}\label{the:eqn:ground-truth}
\begin{cases}
    \divergence\left(\tens C_{t^*} \nabla \vect u^* \right)=0 & \text { in } \Omega, \\ 
    \vect u^*=\vect g & \text { on } \partial \Omega, \\
    \tens C_{t^*} \nabla \vect u^* \cdot \vect n = \vect p & \text { on } \partial \Omega.
\end{cases}
\end{equation}
Here, $\tens C_{t^*}: \bar \Omega \longrightarrow \times_{i=1}^4\R^3$ denotes a parametrised elasticity tensor, depending on the scalar field $t^*: \bar \Omega \longrightarrow \R$. 
The boundary data — namely, the boundary displacement $\vect g \in H^1(\partial \Omega, \R^3)$ and traction $\vect p \in H^1(\partial \Omega, \R^3)$ — are assumed to be known, possibly up to measurement noise.
We consider two specific material models.

\textbf{Isotropic linear elasticity.} In this classical case, the elasticity tensor follows Hooke’s law
\begin{equation}\label{met:eqn:isotropic_law}
\tens C_{t^*} \nabla \vect u^* = \mathbb{C}_{\mu^*} \hat{\nabla} \mathbf{u}=\lambda(\nabla \cdot \mathbf{u}) \mathbf{I}+2 \mu^* \hat{\nabla} \mathbf{u}.
\end{equation}
Lamé's first parameter $\lambda: \bar \Omega \rightarrow \R$ is assumed to be known, and the goal is to recover the shear modulus $\mu^*: \bar \Omega \rightarrow \R$.

\textbf{Anisotropic linear elasticity.} We also consider a structured anisotropic model defined by
\begin{equation}\label{met:eqn:anisotropic_law}
    \mathbf{C}_{t^*}=\mathbf{L}+t^* \mathbf{R},    
\end{equation}
where $\mathbf{L}$ encodes the classical isotropic response via $\mathbf L \nabla \vect u = \mathbb C \symgrad \vect u$, and $\tens R$ represents a directional perturbation of the form $\tens R = (r_j \delta_{jl} \delta_{ik})_{ijkl}$. 
In this setting, we assume knowledge of the Lamé's first parameter $\lambda$, shear modulus $\mu$, and scaling coefficients $r_j: \Omega \longrightarrow \R$ for $j = 1,2,3$, and the unknown parameter to be reconstructed is $t^*: \bar \Omega \rightarrow \R$.

This model choice is motivated by the analysis in \cite{lin_detecting_2018}, which establishes three-spheres inequalities and corresponding propagation properties — key ingredients for the stability analysis of the inverse problem.

A representative example is $\mathbf{R} = (\delta_{fl} \delta_{ik})_{ijkl}$ for a fixed $f \in \{1,2,3\}$, corresponding to a material with a preferred stress direction aligned with a unit vector $\vect e_f$. 
This structure is akin to transverse isotropy, with $\vect e_f$ serving as the fibre direction in a fibre-reinforced material like cardiac muscle tissue \cite{sommer2015biomechanical}.

Furthermore, we put the following assumptions on the boundary data of the ground-truth solution.
\begin{assumption}[Regularity of the boundary data]\label{met:ass:boundary-data}
    We assume that the boundary data satisfies the regularity condition $\vect g \in H^{3/2}\left(\partial \Omega\right)$ for the displacement and $\vect p \in L^{2}\left(\partial \Omega\right)$ for the traction data. Furthermore, the displacement data $\vect g$ should be far from rigid movements, i.e.
    \[
    \Theta(\vect g):=\min \left\{\|\vect g-(\vect a+ \mat W x)\|_{H^{1 / 2}(\partial \Omega)}: \vect a \in \mathbb{R}^n, \mat W \in \mathbb{R}^{n \times n}, \mat W= \mat W^t\right\} \geq \delta_0 > 0.
    \]
\end{assumption}

We state the following inverse problem.
\begin{inverse_problem} \label{met:inv:general_problem}
    Given the boundary displacement $\vect g$, traction data $\vect p$, knowledge of the tensor $\tens C_t$ either as given by \cref{met:eqn:isotropic_law} or by \cref{met:eqn:anisotropic_law} up to the parameter $t$, and the ground-truth solution $\vect u^*$, reconstruct the unknown field parameter $t^*$.
\end{inverse_problem}

\subsection{Stability estimate}
Our main tool for deriving discretisation error estimates and convergence rates is a modified version of the conditional stability result established in \cite[Theorem 3.1]{di_fazio_stable_2017}. The key idea is to introduce a second displacement field $\vect u \in H^1(\Omega, \R^3)$ solving the inhomogeneous boundary value problem:
\begin{equation}\label{eqn:second-system}
\begin{cases}
\divergence \left(\tens C_t \nabla \vect u \right) = \vect f & \text{in } \Omega, \\
\vect u = \vect k & \text{on } \partial \Omega,
\end{cases}
\end{equation}
where $\vect f \in L^2(\Omega, \R^3)$ is a given force term and $\vect k \in H^{3/2}(\partial \Omega)$ denotes the prescribed boundary displacement.

This formulation departs from the setting in \cite{di_fazio_stable_2017} in two essential aspects:
\begin{enumerate}
    \item We allow one of the solutions (here, $\vect u$) to satisfy an inhomogeneous PDE.
    \item Instead of requiring knowledge of the inverse parameter on the boundary, we assume access to the ground-truth pressure $\vect p$ on $\partial \Omega$.
\end{enumerate}
The first adjustment is needed by the structure of the theoretical proofs. The latter one is particularly motivated by practical considerations in in-vivo cardiac tissue identification, where wall pressure measurements are accessible via imaging or catheter-based techniques, whereas direct boundary probing of tissue properties is not feasible.

To derive the stability estimate presented in \Cref{res:cor:stability_estimate}, we now list the structural conditions imposed on the elasticity tensor.

\begin{assumption}\label{met:ass:elasticity_boundedness}
    We require the regularity $t \in \mathcal C ^{3}(\bar \Omega)$ and further assume that
    \begin{enumerate}
        \item $(C_t)_{i,j,k,l} \in \mathcal C^{3}(\bar \Omega)$, $i,j,k,l=1,2,3$, for the components of the elasticity tensor $\tens C_t$,

        \item we have a uniform bound $\norm{t}_{\mathcal C^3(\bar \Omega)} \leq M$,

        \item there is a global $\rho>0$ such that
        \[
        \rho |\mat A|^2 \leq \tens C_t(\vect x) \mat A \cdot \mat A \quad \text{ for all } \mat A \in \operatorname{Sym}(\R^{3 \times 3}) \text{ and all } \vect x \in \bar \Omega.
        \]
    \end{enumerate}
    
\end{assumption}

\begin{remark}
    The high regularity demand is required for the stability results established in \cite{lin_detecting_2018}. For the isotropic case in \cref{met:eqn:isotropic_law}, it suffices to assume $t \in \mathcal C^{1}(\bar \Omega)$. Note, however, that also in view of the existence result of the optimisation problem in \Cref{opt:cor:globalconvergence} below, a control at least on the $H^2(\Omega)$ norm seems to be a reasonable choice. Given our discretisation approach, as developed in \cite{caforio_physics-informed_2024, hofler_physics-informed_2025}, this regularity assumption is appropriate for both settings.
\end{remark}

\begin{remark}
    For the isotropic case, \Cref{met:ass:elasticity_boundedness} means that $\lambda, \mu \in \mathcal C^1(\bar \Omega)$ with the standard convexity condition
    \[
    \mu(\vect x) \geq \mu_0 > 0, \quad 2\mu(\vect x) + 3 \lambda(\vect x) \geq \rho > 0 \quad \text{ for all } \vect x \in \bar \Omega. 
    \]
    For the anisotropic case, one requires $\lambda, \mu, r_1, r_2, r_3 \in \mathcal C^3(\bar \Omega)$, and for the convexity, in addition to the requirements for the isotropic part, that $t \geq t_0>0$ and that $\min_j r_j > 0$. 
\end{remark}

Using the same techniques as in \cite{di_fazio_stable_2017}, we can derive the following proposition, whose derivation is detailed in \Cref{sta:sec}.

\begin{proposition}\label{res:cor:stability_estimate}
    Let \Cref{met:ass:domain} and \Cref{met:ass:boundary-data} hold, $d>0$ be such that $\Omega_d$ is still a connected domain, $(\vect u^*, t^*)$ be solution to \cref{the:eqn:ground-truth}, and $(\vect u, t)$ be solution to \cref{eqn:second-system}. Furthermore, let $t, t^*$ fulfil \Cref{met:ass:elasticity_boundedness} with the a-priori bound $M>0$.
    Then, there exists a constant $C=C(\Omega, M, \norm{\vect k}_{H^{3 / 2}(\partial \Omega)}, \norm{\vect g}_{H^{3 / 2}(\partial \Omega)})>0$ and $\nu \in (0,1)$ such that
    \begin{equation}\label{res:eqn:stability_estimate}
        \left\|t-t^*\right\|_{L^{\infty}\left(\Omega_d\right)} \leq 
        C \Big\{ \norm{\vect f}_{L^2(\Omega)} + \norm{\tens {C}_t \symgrad \vect u \cdot \vect n - \vect p}_{L^2(\partial \Omega)} +
        \norm{\vect u - \vect u^*}_{L^2(\Omega)}^{1/4}
        \Big\}^\nu.
    \end{equation}
    Moreover, the constants $C, \nu$ do not depend on the field parameters $t, t^*$.
\end{proposition}
\begin{proof}
    See \Cref{sta:sec}.
\end{proof}

Unlike the estimate of \cite{di_fazio_stable_2017}, the right-hand side of the stability estimate \cref{res:eqn:stability_estimate} includes both the force term from the second system of \cref{eqn:second-system} and the discrepancy in traction data. 
Following the methodology proposed in \cite{mishra_estimates_2022, mishra_estimates_2023}, the second system of \cref{eqn:second-system} represents the state and parameter to be optimised. 
Ultimately, our goal is for the force term in \cref{eqn:second-system} to vanish, thereby bringing the system into closer alignment with the ground-truth model of \cref{the:eqn:ground-truth}. 
From the standpoint of parameter reconstruction, the stability estimate \cref{res:eqn:stability_estimate} suggests that achieving a good approximation of $t^*$ requires minimising the three terms on its right-hand side. With this in mind, we now proceed with the definition and properties of the continuous optimisation problem.

\subsection{The optimisation problem}\label{sec:opt-error}
To address the inverse problem stated in the \Cref{met:inv:general_problem}, we adopt an all-at-once optimisation framework, where both the parameter field $t$ and the state variable $\vect u$ are treated as optimisation variables. 
The aim is to minimise a loss functional that quantifies the discrepancy in physical residuals and observed data. 
We define the following residual terms:
\begin{equation}\label{opt:eqn:residual}
    \begin{aligned}
        \cR_{PDE}[t](\vect u) &\defeq \divergence \left(\tens C_t \nabla \vect u \right) - \divergence \left(\tens C_{t^*} \nabla \vect u^* \right) & \text{ in } \Omega, \\
        \cR_{BCN}[t](\vect u) &\defeq \tens C_t \nabla \vect u \cdot \vect n - \tens C_{t^*} \nabla \vect u^* \cdot \vect n & \text{ on } \p \Omega, \\
        \cR_{DATA}(\vect u) &\defeq \vect u - \vect u^* & \text{ in } \Omega.
    \end{aligned}
\end{equation}

Note that since the ground-truth solution $(\vect u^*, t^*)$ fulfils \cref{the:eqn:ground-truth}, we can rewrite the first two lines of \cref{opt:eqn:residual} as follows:
\begin{equation*}
    \begin{aligned}
        \cR_{PDE}[t](\vect u) &= \vect f & \text{ in } \Omega, \\
        \cR_{BCN}[t](\vect u) &= \tens C_t \nabla \vect u \cdot \vect n - \vect p & \text{ on } \p \Omega.
    \end{aligned}
\end{equation*}

Based on these residuals, the objective functional to be minimised is given by
\begin{equation}\label{met:eqn:lossfunction}
\begin{aligned}
    \cJ(\vect u, t) &= \lambda_{PDE} \norm{\cR_{PDE}[t](\vect u)}_{L^2(\Omega)}^2
    + \lambda_{BCN} \norm{\cR_{BCN}[t](\vect u)}_{L^2(\partial \Omega)}^2 \\
    & \quad \quad \quad \quad + \lambda_{DATA} \norm{\cR_{DATA}(\vect u)}_{L^2(\Omega)}^2.
\end{aligned}
\end{equation}

Without yet specifying the functional spaces for $t$ and $\vect u$, the continuous-level optimisation problem reads
\begin{equation}\label{met:eqn:opt-problem}
    \inf_{\vect u, t} \cJ(\vect u, t).
\end{equation}

To conclude, we connect the discrepancy in the parameters to the value of the loss functional by introducing the growth function
\[
B(\varepsilon) \coloneqq \begin{cases}
    (\varepsilon/\lambda)^{1/8}, & \text{ for } \varepsilon/\lambda \leq (\varepsilon^*)^2, \\
    (\varepsilon^*)^{1/4}+\sqrt{\varepsilon/\lambda-(\varepsilon^*)^2}, & \text{ for } \varepsilon/\lambda > (\varepsilon^*)^2, \\
\end{cases}
\]
where $0<\epsilon^*<1/2$ is characterised in \Cref{opt:lem:optbound} and 
$\lambda = \min\{\lambda_{PDE}, \lambda_{BCN}, \lambda_{DATA} \}$.

\begin{proposition}\label{res:cor:estimate_param_loss}
    In the setting of \Cref{res:cor:stability_estimate}, it holds that
    \begin{equation}\label{eqn:estimate_param_loss}
        \norm{t - t^*}_{L^\infty(\Omega_d)} \leq CB\left(\cJ(\vect u, t)\right)^\nu.
    \end{equation}
\end{proposition}
\begin{proof}
    This follows by using the stability estimate \cref{res:eqn:stability_estimate} together with the bound from \Cref{opt:lem:optbound}.
\end{proof}

In particular, \cref{eqn:estimate_param_loss} means that, given that we can reach the best optimal loss value $\cJ(\vect u, t)=0$, we get $\vect u = \vect u^*$ in the $L^2(\Omega)$-sense and $t=t^*$ in the $L^\infty(\Omega_d)$-sense. 
Of course, this also means that the candidate spaces we choose for \cref{met:eqn:opt-problem} need to include the ground-truth parameter and state, and that we can also effectively reduce the optimisation error
\[
\mathcal J(\hat{\vect u}, \hat t\, ) - \inf_{\vect u, t} \cJ(\vect u, t),
\]
where $\hat{\vect u}, \hat t$ correspond to the actual computed solution found by our numerical algorithm - still a current challenge when training PINNs.

\section{Regularised inverse problem}

With the minimisation problem and an initial stability estimate in place, we proceed to study the existence of solutions, their convergence properties, and the stability of the method under perturbations caused by noise in the data.

In practice, the measurements of the displacement $\vect u^*$ and boundary traction $\vect p$ are subject to noise. We denote the noisy observations by $\vect u^*(\delta_1)$ and $\vect p (\delta_2)$, which satisfy the error bounds $\norm{\vect u^* - \vect u^*(\delta_1)}_{L^2(\Omega)} \leq \delta_1$ and $\norm{\vect p - \vect p(\delta_2)}_{L^2(\partial \Omega)} \leq \delta_2$.
To incorporate this noise into the loss functional from \cref{met:eqn:lossfunction}, we define modified residuals as follows:
\begin{equation}\label{opt:eqn:residualnoise}
    \begin{aligned}
        \cR_{BCN}^{\delta_2}[t](\vect u) &= \tens C_t \nabla \vect u \cdot \vect n - \vect p(\delta_2), \\
        \cR_{DATA}^{\delta_1}(\vect u) &= \vect u - \vect u^*(\delta_1).
    \end{aligned}
\end{equation}
Assuming for simplicity that the noise levels in both measurements are equal ($\delta \defeq \delta_1 = \delta_2$), the noisy loss functional becomes
\begin{equation}\label{met:eqn:loss-delta}
\begin{aligned}
    \cJ^\delta(\vect u, t) &= \lambda_{PDE} \norm{\cR_{PDE}[t](\vect u)}_{L^2(\Omega)}^2
    + \lambda_{BCN} \norm{\cR_{BCN}^\delta[t](\vect u)}_{L^2(\partial \Omega)}^2 \\
    & \quad \quad \quad \quad + \lambda_{DATA} \norm{\cR_{DATA}^\delta(\vect u)}_{L^2(\Omega)}^2.
    \end{aligned}
\end{equation}

\subsection{Least squares inversion}

At this point, we turn our attention to analysing and clarifying the optimisation problem of \cref{met:eqn:opt-problem}. 
We begin by considering the case with a fixed noise level.

\begin{proposition}\label{opt:cor:globalconvergence}
    Let \Cref{met:ass:domain} and \Cref{met:ass:boundary-data} hold and $t^*$ fulfil \Cref{met:ass:elasticity_boundedness}. Define for given $K,N \geq 0$ the set of admissible solutions as $$\mathcal A=\{ (\vect u, t) \in H^2(\Omega, \R^3) \times W^{2,\infty}(\Omega) \mid \norm{t}_{W^{2,\infty}(\Omega)} \leq K, \norm{\vect u}_{H^2(\Omega)} \leq N \}.$$
    Then, the problem 
    \begin{equation}\label{eqn:noisy-opt-problem}
    \inf_{(\vect u, t) \in \mathcal A} \cJ^\delta(\vect u, t)
    \end{equation}
    admits a solution and any minimising sequence $(\vect u_n, t_n)_n \subset \mathcal A$ of it
    admits a convergent subsequence such that $t_n \rightarrow \tilde t$ in $\mathcal C(\Omega)$ and $\vect u_n \rightharpoonup \tilde{\vect u}$ in $H^2(\Omega, \R^3)$ with $(\tilde{\vect u}, \tilde t) \in \mathcal A$ being solution to \cref{eqn:noisy-opt-problem}.\\
    If we additionally assume that the limit point $\tilde t$ fulfils \Cref{met:ass:elasticity_boundedness}, then we have the a-posteriori estimate
    \begin{equation}\label{opt:eqn:posteriori-estimate}
    \begin{aligned}
    \norm{\tilde t - t^*}_{L^\infty(\Omega_d)} &\leq C B\left( \inf_{(\vect u, t) \in \mathcal A} \cJ^\delta(\vect u, t) + 2 \max\{\lambda_{BCN}, \lambda_{DATA}\} \delta^2 \right)^\nu \\
    &\leq C B\left( 4 \max\{\lambda_{BCN}, \lambda_{DATA}\} \delta^2 \right)^\nu,
    \end{aligned}
    \end{equation}
    where the constants $C>0$, $\nu>0$ are those from the stability estimate in \Cref{res:cor:stability_estimate}.
\end{proposition}
\begin{proof}
    First note that for all  $\vect u, t$, the objective functional is bounded from below by $\cJ^\delta(\vect u, t) \geq 0$ and proper on $\mathcal A$. As a consequence, we have $0 \leq \displaystyle \inf_{(\vect u, t) \in \mathcal A} \cJ^\delta(\vect u, t) < \infty$.
    Hence, we can take an infimal sequence $(\vect u_n, t_n)_n \subset \mathcal A$ such that
    \[
    \cJ^\delta(\vect u_n, t_n) \rightarrow \inf_{(\vect u, t) \in \mathcal A} \cJ^\delta(\vect u, t) \quad \text{ for } n \rightarrow \infty.
    \]
    Since $\vect u_n$ is bounded in $H^2(\Omega, \R^3)$, we can extract a weakly convergent subsequence, denoted by the same indices, such that $\vect u_n \rightharpoonup \tilde {\vect u}$ in $H^2(\Omega, \R^3)$. Similarly, since $t_n$ is bounded in $W^{2,\infty}(\Omega)$, we can extract a weak* convergent subsequence $t_n \overset{*}{\rightharpoonup} \tilde t$ in $W^{2,\infty}(\Omega)$. Furthermore, since $\mathcal A$ is weakly closed in the state and weak* closed in the parameter, we get $(\tilde {\vect u},\tilde t) \in \mathcal A$. Then, by the boundedness of the sequence $(t_n)_n$ in $L^2(\Omega)$ due to the embedding $L^\infty \hookrightarrow L^2$, and since the dual of $L^1$ is $L^\infty$, and $L^1 \cap L^2$ is dense in $L^2$, we also get $t_n \rightharpoonup \tilde t$ in $H^{2}(\Omega)$.

    Then, by \Cref{met:ass:domain}, we have the continuous embedding
    $W^{2, \infty}(\Omega) \hookrightarrow \mathcal C^1(\bar \Omega)$ (see e.g. \cite[Theorem 10.5]{alt_linear_2016}) which translates into a bound
    \[
    \norm{t}_{C^{1}(\Omega)} \leq K,
    \]
    and thus we can apply the Arzelà-Ascoli theorem, see e.g. \cite[Theorem 4.25]{brezis_functional_2011}, to extract a globally convergent subsequence $t_n \rightarrow \hat t$ in $\mathcal C(\Omega)$. Due to the same embedding, in particular because 
    $\mathcal C^1(\bar \Omega) \hookrightarrow W^{2, \infty}(\Omega) \hookrightarrow H^2(\Omega)$,
    we have $\hat t = \tilde t$ in $\mathcal C(\bar \Omega)$.

    Using the continuity properties given in \Cref{con:sec}, in particular \Cref{con:lem:iso-continuity}, \Cref{con:lem:aniso-continuity}, and \Cref{con:lem:boundary-continuity}, we get $\mathcal J$ and hence $\cJ^\delta$ is lower semi-continuous with respect to weak convergence in $H^2(\Omega, \R^3) \times H^2(\Omega)$, hence proving that $(\tilde {\vect u}, \tilde \mu)$ minimises $\cJ^\delta$ over $\mathcal A$.

    Now, by observing that $\cJ \leq \cJ^\delta + 2 \max\{\lambda_{BCN}, \lambda_{DATA}\} \delta^2$ pointwise, and by applying the stability estimate of \cref{eqn:estimate_param_loss} together with the weak lower semi-continuity of $\cJ$, we obtain
    \[
    \begin{aligned}
    \norm{\tilde t - t^*}_{L^\infty(\Omega_d)}
    &\leq CB\left( \cJ (\tilde{\vect u}, \tilde t) \right) ^\nu \\
    &\leq CB\left(\liminf_{n \rightarrow \infty} \cJ(\vect u_n, t_n)\right)^\nu \\
    &\leq CB\left(\liminf_{n \rightarrow \infty} \cJ^\delta(\vect u_n, t_n) + 2\max\{\lambda_{BCN}, \lambda_{DATA}\}\delta^2\right)^\nu \\
    &\leq CB\left( \inf_{(\vect u, t) \in \mathcal A} \cJ^\delta(\vect u, t) + 2\max\{\lambda_{BCN}, \lambda_{DATA}\}\delta^2\right)^\nu.
    \end{aligned}
    \]
    The last inequality can be further estimated from above by choosing any element from $\mathcal A$. 
    In particular, to obtain the last inequality in \cref{opt:eqn:posteriori-estimate}, we can just insert $(\vect u^*, t^*)$.
\end{proof}

Clearly, imposing the regularity condition of \Cref{met:ass:elasticity_boundedness} on an a-priori unknown limit point is not practically viable. 
However, in the asymptotic case where the noise level tends to zero, this requirement corresponds to a constraint on the true solution itself — an assumption that may be acceptable in practical scenarios.

\begin{theorem}\label{opt:thm:noisyexistence}
    Let \Cref{met:ass:domain} and \Cref{met:ass:boundary-data} hold. Denote by $\mathcal A$ the candidate space as given in \Cref{opt:cor:globalconvergence} and let $(\vect u^\delta, t^\delta) \in \mathcal A$ denote a solution to
    \begin{equation*}
    \inf_{(\vect u, t) \in \mathcal A} \cJ^\delta(\vect u, t),
    \end{equation*}
    where $\mathcal J^\delta(\vect u, t)$ is given by \cref{met:eqn:loss-delta} and the noise level $\delta>0$ is given a-priori.
    Then, for the noise level $\delta \rightarrow 0$, we have for a subsequence, also denoted by $(\vect u^\delta, t^\delta)$, that
    \begin{equation}\label{eqn:conv-noisy-to-zero}
    \begin{aligned}
        \vect u^\delta &\rightharpoonup \hat{\vect u} \quad \text{ in } H^2(\Omega, \R^3), \\
        t^\delta &\rightarrow \hat t \quad \phantom{.}\text{ in } \mathcal C(\Omega),
    \end{aligned}
    \end{equation}
    where $(\hat {\vect u}, \hat t) \in \mathcal A$ is a solution to 
    \begin{equation*}
    \inf_{(\vect u, t) \in \mathcal A} \cJ(\vect u, t).
    \end{equation*}

    In particular, if the limit point $\hat t$ fulfils \Cref{met:ass:elasticity_boundedness}, we get strong convergence $t^\delta \rightarrow t^*$ in $L^\infty(\Omega_d)$ by \Cref{opt:cor:globalconvergence}.
\end{theorem}
\begin{proof}
    Since $\cJ^\delta(\vect u, t) \leq \cJ(\vect u, t)+2\max\{\lambda_{BCN}, \lambda_{DATA}\}\delta^2$ for all $(\vect u, t) \in \mathcal A$, $\vect u(\delta) \rightarrow \vect u^*$ and $\vect p(\delta) \rightarrow \vect p$ strongly in $L^2(\Omega)$ and $L^2(\partial \Omega)$ respectively, we have $\cJ^\delta \rightarrow \cJ$ pointwise as $\delta \rightarrow 0$.

    Using the same embeddings as in the proof of \Cref{opt:cor:globalconvergence}, we get the proposed convergence (up to subsequences) as stated in \cref{eqn:conv-noisy-to-zero}. At this point, we denote the corresponding limit point by $(\tilde{\vect u}, \tilde t)$ since it is subject of the proof to show that it solves the noiseless problem.
    
    Furthermore, since $(\vect u^\delta, t^\delta)$ are minimisers of the noisy loss functional $\cJ^\delta$, we have 
    \[
    \cJ^\delta(\vect u^\delta, t^\delta) \leq \cJ^\delta(\hat{\vect u}, \hat t)
    \leq \cJ(\vect u^*, t^*)+2\max\{\lambda_{BCN}, \lambda_{DATA}\}\delta^2.
    \]
    This implies that
    \[
    \limsup_{\delta \rightarrow 0} \cJ^\delta(\vect u^\delta, t^\delta) \leq \cJ(\vect u^*, t^*)=0.
    \]
    Also, using $a^2-b^2=(a-b)(a+b)$ together with Cauchy-Schwarz inequality, the reverse triangle inequality and the boundedness assumption on the set $\mathcal A$ together with the trace theorem (see e.g. \cite[Theorem 2.21]{steinbach_numerical_2008}), we can estimate, for a given noise level $\eta>0$,
    \[\begin{aligned}
    \left| \cJ^\eta(\vect u^\delta,t^\delta) - \cJ(\vect u^\delta,t^\delta) \right|
    &\leq \lambda_{BCN} \int_{\p \Omega} \left| |\mathbb{C}_{t^\delta} \symgrad \vect u^\delta \cdot \vect n - \vect p (\eta)
    |^2 - | \mathbb{C}_{t^\delta} \symgrad \vect u^\delta \cdot \vect n - \vect p |^2
    \right| ds \\ 
    & \quad \quad \quad \quad + \lambda_{DATA}
    \int_\Omega \left| |\vect u^\delta - \vect u^*(\eta)|^2 - |\vect u^\delta - \vect u^*|^2 \right| dx \\
    &\leq C(\lambda_{BCN}, \lambda_{DATA}, \Omega, M,N) \left(
    \norm{ \vect p (\eta) - \vect p }_{L^2(\p \Omega} + \norm{ \vect u^*(\eta) - \vect u^* }_{L^2(\Omega} \right),
    \end{aligned}\]
    which shows that $\cJ^\eta(\vect u^\delta,t^\delta) \rightarrow \cJ(\vect u^\delta,t^\delta)$ uniformly as $\eta \rightarrow 0$. With this, we can interchange the loss functionals in the limit, meaning that $\liminf_{\delta \rightarrow 0} \cJ( \vect u^\delta, t^\delta)
    = \liminf_{\delta \rightarrow 0} \cJ^\delta( \vect u^\delta, t^\delta)$.
    Together with the weak lower semi-continuity of $\cJ$, we get 
    \[
    \cJ(\tilde {\vect u}, \tilde t) \leq \liminf_{\delta \rightarrow 0} \cJ( \vect u^\delta, t^\delta)
    \leq \liminf_{\delta \rightarrow 0} \cJ^\delta( \vect u^\delta, t^\delta)
    \leq \limsup_{\delta \rightarrow 0} \cJ^\delta( \vect u^\delta, t^\delta) \leq \cJ(\vect u^*, t^*)=0.
    \qedhere
    \]
\end{proof}

This is a standard property that one would expect from a sufficiently regularised optimisation problem, as discussed, for instance, in \cite{bredies_higher-order_2020}. 
Furthermore, the formulation can be easily adapted to account for displacement and traction data corrupted by noise of varying intensity.

However, the introduction of the norm restrictions $\norm{t}_{W^{2,\infty}(\Omega)} \leq K, \norm{\vect u}_{H^2(\Omega)} \leq N $ is somewhat artificial. 
Of course, $\vect u \in H^2(\Omega, \R^3)$ seems to be a reasonable requirement. 
Also note that in view of the stability estimate, we have $t \in \mathcal C^{k}(\bar \Omega) \hookrightarrow W^{2, \infty}(\Omega)$ if $k \geq 1$ (e.g. by \cite[Theorem 10.5]{alt_linear_2016}).

\section{Discretisation}

We proceed by introducing finite-dimensional, typically conforming, ansatz spaces $X_\Theta \subset H^2(\Omega, \R^3)$ and $V_H \subset H^2(\Omega)$ to approximate the state and the parameter, respectively. The discrete approximations are denoted by $\vect u_\theta \in X_\Theta$ and $t_\eta \in V_H$. Accordingly, the residuals are reformulated as:
\begin{equation}\label{eqn:residual-interior}
    \begin{aligned}
        \cR_{PDE}[\mvect \eta](\mvect \theta) &= \divergence \left( \mathbb C_{t_\eta} \symgrad \vect u_\theta \right), \\
        \cR_{BCN}[\mvect \eta](\mvect \theta) &= \mathbb{C}_{t_\eta} \symgrad \vect u_\theta \cdot \vect n - \vect p, \\
        \cR_{DATA}(\mvect \theta) &= \vect u_\theta - \vect u^*.
    \end{aligned}
\end{equation}

Hence, the discretised loss function is given by:
\begin{equation*}
\begin{aligned}
    \cJ(\mvect \theta,\mvect \eta) &= \lambda_{PDE} \norm{\cR_{PDE}[\mvect \eta](\mvect \theta)}_{L^2(\Omega)}^2
    + \lambda_{BCN} \norm{\cR_{BCN}[\mvect \eta](\mvect \theta)}_{L^2(\partial \Omega)}^2 \\
    & \quad \quad \quad \quad + \lambda_{DATA} \norm{\cR_{DATA}(\mvect \theta)}_{L^2(\Omega)}^2.
\end{aligned}
\end{equation*}

At this stage, we remain within a continuous-in-space setting, as our primary interest lies in analytical norms rather than their discrete numerical counterparts. Our initial aim is to derive a generic bound on the loss function that is, in principle, independent of any specific discretisation scheme.

Discretisation errors are subsequently characterised through a stability estimate and the aforementioned loss function bounds. 
To ensure the generality of our analysis, we introduce an additional continuity assumption, which is verified for the systems of \cref{met:eqn:isotropic_law} and \cref{met:eqn:anisotropic_law}, as discussed in \Cref{con:lem:joint_bound}.

\begin{assumption}\label{res:ass:joint_bound}
    We assume that for all $\vect u_1, \vect u_2 \in H^2(\Omega, \R^3)$ and $t_1, t_2\in H^1(\Omega)$, we have the estimate
    \[
    \norm{\tens C_{t_1} \nabla \vect u_1 - \tens C_{t_2} \nabla \vect u_2}_{H^1(\Omega)} \leq C_b \left( \norm{\vect u_1 - \vect u_2}_{H^2(\Omega)} + \norm{t_1 - t_2}_{H^1(\Omega)} \right),
    \]
    with a constant $C_b = C_b(\norm{\vect u_1}_{H^2(\Omega)}, \norm{\vect u_2}_{H^2(\Omega)}, \norm{t_1}_{H^1(\Omega)}, \norm{t_2}_{H^1(\Omega)})$ that depends continuously on its arguments.
\end{assumption}

We also note that \Cref{met:cor:dsrnrate} below requires the function subject to approximation to lie in $H^3$ in order to obtain a positive convergence rate exponent. 
However, this regularity assumption is likewise necessary for the mesh-based approximation result stated in \Cref{met:thm:fem-approx} to ensure control in the $H^2(\Omega)$ norm. 
Consequently, for the discretisation analysis, we assume at least $\vect{u}^* \in H^3(\Omega, \mathbb{R}^3)$ and $t^* \in H^3(\Omega)$ to obtain meaningful bounds.

We begin by stating the following upper bound on the loss functional, which follows directly from \Cref{res:ass:joint_bound}.

\begin{lemma}\label{pro:lem:lossbound}
    Assume that $\vect u \in H^2(\Omega, \R^3)$, $t \in H^1(\Omega)$. 
    Then, we can bound the loss function of \cref{met:eqn:lossfunction} by
    \begin{equation*}
    \cJ(\vect u, t) \leq
    C \left( \norm{\vect u - \vect u^*}_{H^2(\Omega)}^2 + \norm{t - t^*}_{H^1(\Omega)}^2 \right),
    \end{equation*}
    with a constant $C=C(\Omega, M, \norm{\vect u^*}_{H^2(\Omega)})>0$.
\end{lemma}
\begin{proof}
    \Cref{res:ass:joint_bound} directly translates into a bound for the PDE residual since $H^1(\Omega,\R^3) \hookrightarrow H(\operatorname{div};\Omega,\R^3)$. 
    Using the trace theorem, e.g. \cite[A 8.6]{alt_linear_2016}, we get $H^1(\Omega) \hookrightarrow L^2(\p \Omega)$ leading to the same estimate for the boundary residual with the additional factor $C_{tr}>0$ from the trace theorem. 
    Overall, we get the estimate
    \begin{equation*}
    \cJ(\vect u, t) \leq
    (C_b + C_{tr} C_b + 1) 
    \left( \norm{\vect u - \vect u^*}_{H^2(\Omega)}^2 + \norm{t - t^*}_{H^1(\Omega)}^2 \right). \qedhere
    \end{equation*}
\end{proof}

\subsection{Neural network approximations}
In the PINN framework, both the state variable and the unknown parameter are approximated using neural networks. Hence, the question of general approximation results in Sobolev spaces arises naturally. The literature on these approximation results has increased vastly in the last years. In \cite{yang_nearly_2023}, so-called deep super ReLU networks are introduced with approximation rates in Sobolev norms. The work \cite{morina_growth_2024} shows parameter growh rates for ReLU networks. Since \Cref{met:ass:elasticity_boundedness} requires high regularity on the employed networks, we carry out the analysis with tanh networks, where \cite{de_ryck_approximation_2021} showed corresponding approximation results in Sobolev norms.

We start with the notation of a standard MLP, i.e. a feedforward network composed of multiple layers of neurons representing affine transformations and with the same non-linear activation function in-between them.
Formally, let the input to the network be denoted by $ \vect x ^{0} \in \R^{n_0}$. 
For each layer $l=1,2,...,L-1$, the $l$-th forward pass is defined via
\[
\vect x^{l} = \sigma\left(\mat {W}^{l} \vect {x}^{l-1} + \vect {b}^{l}\right),
\]
where $\mat {W}^{l} \in \mathbb{R}^{n_l \times n_{l-1}}$ is the weight matrix, $\vect {b}^{l} \in \mathbb{R}^{n_l}$ the bias vector, and $\sigma$ is an element-wise applied, nonlinear activation 
function.
The final output of the network is given by
\[
\vect {y} = \mat {W}^{L} \vect {x}^{L-1} + \vect {b}^{L}.
\]

A given MLP is therefore fully characterised by its weights, biases, and its activation function. 
We summarise this information in the vector
\begin{equation*}
    \mathcal{A} = \left( (\mat {W}^{1}, \vect {b}^{1}), ..., (\mat {W}^{L}, \vect {b}^{L}) \mid \sigma \right),
\end{equation*}
and identify therefore a MLP with $\mathcal{A}$.

Accordingly, we define the corresponding network function of the MLP as
\begin{equation*}
    R(\mathcal{A}): \R^{n_0} \longrightarrow \R^{n_L}, \quad \vect {x}^{0} \mapsto \vect {y}.
\end{equation*}

We measure the size of a MLP $\mathcal{A}$ by the vector $(L(\mathcal{A}), N(\mathcal{A}))$, where $L(\mathcal{A})=L-1$ denotes the depth, i.e. number of hidden layers or the count of weight and bias tuples $(\mat {W}^{(l)}, \vect {b}^{(l)})$, and $N(\mathcal{A})=\displaystyle \max_{l=1...L-1} n_l$ the width, i.e. the maximal size over all hidden layers. By convenience, we denote by $S(\mathcal{A})=\sigma$ its activation function.

For a given depth $L\in\N$, width $N\in\N$, and activation function $\sigma$, we define
\begin{equation*}
    \mathcal{F}(L,N, \sigma) \defeq \left\{ \mathcal{A} \text{ MLP } \mid L(\mathcal{A}) \leq L, N(\mathcal{A}) \leq N, S(\mathcal{A})=\sigma\right\},
\end{equation*}
to be a given class of networks.

As discussed in the previous section, the bound in \cref{opt:eqn:posteriori-estimate} serves as a practical criterion for estimating the parameter error. However, this needs a-priori assumptions on the limit point $\tilde t$. 
Fortunately, by imposing additional explicit constraints into the neural network architecture, we can make the whole sequence $(t_n)$ of \Cref{res:cor:stability_estimate} fulfil the regularity condition (1) and convexity condition (3) of \Cref{met:ass:elasticity_boundedness} immediately. In particular, using tanh networks makes the network smooth whereas constraining the output to be strictly positive, e.g. by adding a rescaled tanh activation to the parameter network output, leads to the sought convexity.
Regarding the uniform bound (2) of \Cref{met:ass:elasticity_boundedness}, we note that the $n$-th derivative of a single layer network $R(\vect x) = \sigma(\mat W \vect x + \vect b)$ is given by
\[
R^{(n)}(\vect x) = \sigma^{(n)}(\mat W \vect x + \vect b) \odot \left( \otimes_n \mat W \right), 
\]
where $\odot \left( \otimes_n \mat W \right)$ denotes a tensor product of $n$ outer products of rows of $\mat W$. 
Importantly, since $\vect x \in \bar \Omega$ is bounded, this means that we just need enough regularity on the activation function $\sigma$ and ensure boundedness of the weights.

We now state the key theoretical result from \cite[Theorem 5.1]{de_ryck_error_2022}, which establishes the approximation capabilites of tanh networks in Sobolev spaces.

\begin{corollary}[Sobolev space approximation]\label{met:cor:dsrnrate}
    Let $k \in \N$, $\Omega=(0,1)^3$, and $f \in W^{k, \infty}(\Omega)$. Then, for a given $\tilde N \in \N$ large, there exists a tanh network $R(\mathcal A) \in \mathcal F(2,N,\sigma)$ with $N=\mathcal O(\tilde N^3)$ such that
    \[
    \norm{f-R(\mathcal A)}_{W^{2,\infty}(\Omega)} = \mathcal O \left( \frac{\log^2 \tilde N}{\tilde N^{k-2}} \right).
    \]
\end{corollary}
\begin{proof}
    This is a simplification of \cite[Theorem 5.1]{de_ryck_error_2022}. In particular, we use the original bound
    \[
    \norm{f-R(\mathcal A)}_{W^{2,\infty}(\Omega)} \leq 27(1+\delta)6^6 \log^2(\beta \tilde N^{k+5}) \frac{\mathcal C(3,2,k,f)}{\tilde N^{k-2}},
    \]
    where $\delta>0, \beta>1$, and $\mathcal C(3,2,k,f)$ are defined in \cite{de_ryck_error_2022}.
\end{proof}

\subsubsection{Discretisation errors with PINNs}

This section examines the impact of neural network-based discretisation on the optimisation problem. 
Specifically, we aim at deriving a bound on the interior error of the sought parameter in terms of the network sizes. 
Such bounds are referred to by different names in the literature: in the context of learning theory, they are often called \emph{generalisation errors}, as discussed in \cite{mishra_estimates_2022, mishra_estimates_2023}, while in classical FE analysis they are known as \emph{best approximation errors}.

We now invoke the general neural network approximation results provided in \Cref{met:cor:dsrnrate}.

\begin{theorem}[Approximation error for the parameter, neural-network based]\label{thm:approx-error-shear-modulus}
    Let \Cref{met:ass:domain}--\ref{met:ass:elasticity_boundedness} and \Cref{res:ass:joint_bound} hold and assume the (possibly increased) regularity $\vect u^* \in H^k(\Omega, \R^3)$ and $t^* \in H^k(\Omega)$ with $k \geq 3$.
    Let $\varepsilon > 0$ be a given level of tolerance. Then, there exists a tanh network class $\mathcal F(2,N, \sigma)$ such that for the minimisers $\hat \bu _\theta \in \mathcal F(2,N, \sigma)^3$, $\hat t_\eta \in \mathcal F(2,N, \sigma)$ of
    \[
    \min_{\substack{\bu _\theta \in \mathcal F(2,N, \sigma)^3 \\ t_\eta \in \mathcal F(2,N, \sigma)}} \mathcal J^\delta(\boldsymbol \theta, \boldsymbol \eta),
    \]
    we achieve
    \begin{equation*}
        \left\|\hat t_\eta-t^*\right\|_{L^{\infty}\left(\Omega_d\right)} \leq \varepsilon + CB(4\max\{\lambda_{BCN}, \lambda_{DATA}\}\delta^2)^\nu,
    \end{equation*}
    where the constants $C,\nu>0$ are from \Cref{res:cor:stability_estimate}.
    
    In particular, we get the worst-case bound
    \begin{equation*}
        N = \Omega \left( \epsilon ^{- 24/(\nu(k-2)-2\nu\rho)} \right),
    \end{equation*}
    where $\rho>0$ for the case $k=3$ and $\rho=0$ if $k>3$.
\end{theorem}
\begin{proof}
    We focus first on the noiseless problem.
    By \Cref{met:cor:dsrnrate}, there exist tanh networks $\vect u_\theta \in \mathcal F(2,N,\sigma)$, $t_\eta \in \mathcal F(2,N, \sigma)$ such that
    \begin{align*}
        \norm{\vect u_\theta - \vect u^*}_{H^2(\Omega)} &= \mathcal O \left( 
    \frac{\log^2 \tilde N}{\tilde N^{k-2}}
    \right)=\mathcal O \left( \tilde N ^{-(k-2)/2+\rho}\right), \\
    \norm{t_\eta - t^*}_{H^2(\Omega)} &= \mathcal O \left( 
    \frac{\log^2 \tilde N}{\tilde N^{k-2}}
    \right)=\mathcal O \left( \tilde N ^{-(k-2)/2+\rho}\right),
    \end{align*}
    where $\rho>0$ for the case $k=3$ and $\rho=0$ if $k>3$.

    In particular, using \Cref{pro:lem:lossbound}, we can estimate
    \begin{equation}\label{opt:eqn:approxerrorestimate}
        \cJ(\mvect \theta,\mvect \eta) \leq
        C \left( \norm{\vect u_\theta - \vect u^*}_{H^2(\Omega)}^2 + \norm{t_\eta - t^*}_{H^1(\Omega)}^2 \right) \leq \tilde{\varepsilon},
    \end{equation}
    with a given tolerance level $\tilde \varepsilon > 0$, where we have  
    \begin{equation*}
        \tilde N = \Omega \left( \tilde \epsilon ^{- 1/((k-2)-2\rho)} \right).
    \end{equation*}
    
    Applying now Lemma \ref{opt:lem:optbound} with $\tilde{\epsilon}$, we obtain

    \begin{equation*}
        \norm{\cR_{PDE}[\mvect \eta](\mvect \theta)}_{L^2(\Omega)} + \norm{\cR_{BCN}[\mvect \eta](\mvect \theta)}_{L^2(\partial \Omega)} +
    \norm{\cR_{DATA}(\mvect \theta)}_{L^2(\Omega)}^{1/4}
    \leq B(\tilde{\epsilon}),
    \end{equation*}
    such that by the stability estimate of \cref{res:eqn:stability_estimate}, we have 

    \begin{equation}\label{eqn:bound-noiseless-net}
        \left\|t_\eta-t^*\right\|_{L^{\infty}\left(\Omega_d\right)} \leq C B(\tilde{\varepsilon})^\nu,
    \end{equation}
    where $C>0$ combines the previous constant from \cref{opt:eqn:approxerrorestimate} with the one from the stability estimate \cref{res:eqn:stability_estimate}.
    
    The conclusion follows by choosing $\tilde{\varepsilon}$ small enough such that $C B(\tilde{\varepsilon})^\nu \leq \varepsilon$. In particular, since $B(\tilde \epsilon)= \mathcal O (\tilde \epsilon ^{1/8})$, we need to choose
    \begin{equation*}
        \tilde N = \Omega \left( \epsilon ^{- 8/(\delta(k-2)-2\nu\rho)} \right).
    \end{equation*}
    In the worst case, the bound $N=\mathcal O(\tilde N^3)$ is asymptotically sharp such that we get
    \begin{equation*}
        N = \Omega \left( \epsilon ^{- 24/(\nu(k-2)-2\nu\rho)} \right).
    \end{equation*}

    Now, we also want to consider the noisy problem. We note that
    \begin{equation}\label{eqn:loss-estimate-noisy-noiseless}
        \mathcal J(\hat \bu_\theta, \hat t_\eta) \leq \mathcal J(\bu_\theta, t_\eta)+4\max\{\lambda_{BCN}, \lambda_{DATA}\}\delta^2.
    \end{equation}
    Note that $B(\cdot)^\nu$ is concave, monotonically increasing, satisfies $B(0)=0$, and therefore is subadditive. This makes it possible to replace $t_\eta$ with $\hat t_\eta$ in \cref{eqn:bound-noiseless-net} such that we get the decomposition
    \begin{align*}
        \left\|\hat t_\eta-t^*\right\|_{L^{\infty}\left(\Omega_d\right)} &\leq CB\left( \mathcal J(\hat \bu_\theta, \hat t_\eta) \right)^\nu \\
        &\leq CB\left( \mathcal J(\bu_\theta, t_\eta)+4\max\{\lambda_{BCN}, \lambda_{DATA}\}\delta^2 \right)^\nu \\
        &\leq C B(\tilde{\varepsilon})^\nu + CB(4\max\{\lambda_{BCN}, \lambda_{DATA}\}\delta^2)^\nu,
    \end{align*}
    where $\hat \bu_\theta, \hat t_\eta \in \mathcal F(2,N,\sigma)$ are in the same network class as defined for the noiseless solution.
\end{proof}

For the sake of completeness, we also incorporate the quadrature error, which is caused by the approximation of the analytic norms, into our analysis. 

Assume that the given quadrature rules satisfy
\begin{equation}\label{opt:eqn:quadrature-shear-modulus}
\begin{aligned}
    \left| \norm{\mathcal{R}_{PDE}[\mvect{\eta}](\mvect{\theta})}_{L^2(\Omega)}^2 
    - Q_{N_1}^\Omega \left[ \left| \mathcal{R}_{PDE}[\mvect{\eta}](\mvect{\theta}) \right|^2 \right] \right| 
    &= \mathcal{O}(N_1^{-\alpha_1}), \\
    \left| \norm{\mathcal{R}_{BCN}[\mvect{\eta}](\mvect{\theta})}_{L^2(\partial \Omega)}^2 
    - Q_{N_2}^{\partial \Omega} \left[ \left| \mathcal{R}_{BCN}[\mvect{\eta}](\mvect{\theta}) \right|^2 \right] \right| 
    &= \mathcal{O}(N_2^{-\alpha_2}), \\
    \left| \norm{\mathcal{R}_{DATA}(\mvect{\theta})}_{L^2(\Omega)}^2 
    - Q_{N_3}^\Omega \left[ \left| \mathcal{R}_{DATA}(\mvect{\theta}) \right|^2 \right] \right| 
    &= \mathcal{O}(N_3^{-\alpha_3}),
\end{aligned}
\end{equation}
where the collocation points used in the quadrature rules are summarised by the tuple \( \mathcal{S} = (\mathcal{S}_{PDE}, \mathcal{S}_{BCN}, \mathcal{S}_{DATA}) \).

We define the resulting empirical loss function as
\begin{align*}
    \mathcal{J}_N(\mvect{\theta}, \mvect{\eta} \mid \mathcal{S}) &=
    \lambda_{PDE} \, Q_{N_1}^\Omega \left[ \left| \mathcal{R}_{PDE}[\mvect{\eta}](\mvect{\theta}) \right|^2 \right] 
    + \lambda_{BCN}\,  Q_{N_2}^{\partial \Omega} \left[ \left| \mathcal{R}_{BCN}[\mvect{\eta}](\mvect{\theta}) \right|^2 \right] \\
    & \quad \quad \quad \quad + \lambda_{DATA} \, Q_{N_3}^\Omega \left[ \left| \mathcal{R}_{DATA}(\mvect{\theta}) \right|^2 \right].
\end{align*}

\begin{corollary}[A-posteriori error]
    Assume that we have two networks $\vect u_\theta, t_\eta \in \mathcal{F}(L,N, \sigma)$ such that 
    \[
    \cJ_N(\mvect \theta,\mvect \eta | \mathcal{S}) \leq \varepsilon,
    \] for a given set of collocation points $\mathcal{S}$ and tolerance level $\varepsilon>0$.
    Define the total quadrature error by
    \begin{equation*}
        Q(\mathcal{S}) \coloneq \lambda_{PDE} N_1^{-\alpha_1} + \lambda_{BCN} N_2^{-\alpha_2} + \lambda_{DATA} N_3^{-\alpha_3}.
    \end{equation*}
    Then, we can bound the interior error for the parameter by
    \begin{equation*}
        \left\|t_\eta-t^*\right\|_{L^{\infty}\left(\Omega_d\right)} \leq K B\left( \varepsilon + Q(\mathcal{S}) \right) ^{\delta},
    \end{equation*}
    where the constant $K>0$ also depends on the quadrature rule.
\end{corollary}
\begin{proof}
    By (\ref{opt:eqn:quadrature-shear-modulus}), we have 
    \begin{equation*}
        \mathcal{J}(\mvect \theta, \mvect \eta) \leq \cJ_N(\mvect \theta,\mvect \eta | \mathcal{S}) 
        + K Q(\mathcal{S}) \leq \varepsilon + C Q(\mathcal{S}),
    \end{equation*}
    where $K>0$ depends on the quadrature rule.
    Hence, by the estimate of \cref{res:eqn:stability_estimate}, we have 
    \begin{equation*}
        \left\|t_\eta-t^*\right\|_{L^{\infty}\left(\Omega_d\right)} \leq CB(\varepsilon + KQ(\mathcal{S}))^\delta \leq K B\left( \varepsilon + Q(\mathcal{S}) \right) ^{\delta},
    \end{equation*}
    where the constant $K>0$ now combines the constant from the stability estimate of \cref{res:eqn:stability_estimate} as well as the constant from the quadrature rule.
\end{proof}

\subsection{Mesh-based approximations}
This section introduces a mesh-based function space approximation as used in FE methods. For this case, we consider a quasi-uniform triangulation of the domain which, for simplicity, will be the same both for the state and the parameter space. 

We use the following interpolation result from \cite[Theorem 4.28]{grossmann_numerical_2007}.

\begin{theorem}\label{met:thm:fem-approx}
    Let $(T_j)_j$ be a quasi-uniform triangulation of $\Omega$ with mesh size $h=\max_j h_j$, where $h_j>0$ denotes the diameter of $T_j$. 
    Further, for $u \in H^{k}(\Omega)$, denote by $\mathcal P_{k-1}u$ its projection onto the space of piecewise polynomials of degree at most $k-1$ defined over the triangulation $(T_j)_j$. 
    Additionally, for $0 \leq r \leq k$, assume that $\mathcal P_{k-1} u \in H^r(\Omega)$.
    Then, there exists a constant $c>0$ such that
    \[
    \norm{u-\mathcal P_{k-1}u}_{H^r(\Omega)} \leq c h ^ {k-r} |u|_{H^{k}(\Omega)}.
    \]
\end{theorem}

For later purposes, we need $\mathcal P_{k-1} u \in H^2(\Omega)$, which requires $\mathcal C^1$-conforming FE methods. 
A classical approach is to use the TUBA family \cite{argyris_tuba_1968}, which was originally introduced for 2D problems and whose extension to three dimensions requires high computational efforts. 
B-splines provide an attractive alternative for the function space approximation, see e.g. \cite{hollig_finite_2003}.
Nonetheless, they suffer from high curvature if not tuned properly.
The work \cite{karnakov_solving_2023} describes how classical FE like ansatz spaces can be used in a least-squares optimisation framework. 
In particular, the authors use a finite-volume discretisation. 
However, we emphasise that this does not fall into this framework due to the lack of the required regularity.

\subsubsection{Mesh-based discretisation errors}\label{sec:mesh-based-error}

We also want to state an alternative version of \Cref{thm:approx-error-shear-modulus} using \Cref{met:thm:fem-approx} instead of the approximation with neural networks.

\begin{theorem}[Approximation error for the parameter, mesh-based]\label{thm:approx-error-shear-modulus-mesh}
    Let \Cref{met:ass:domain}--\ref{met:ass:elasticity_boundedness} and \Cref{res:ass:joint_bound} hold and assume the (possibly increased) regularity $\vect u^* \in H^k(\Omega, \R^3)$ and $t^* \in H^k(\Omega)$ with $k \geq 3$.
    Let $\varepsilon > 0$ be a given level of tolerance. Then, there exists a mesh size $h>0$ and a corresponding FE space $\mathcal P_{k-1}H^k(\Omega)$ such that for the minimisers $\hat \bu_\theta \in P_{k-1}H^k(\Omega)^3$, $\hat t_\eta \in P_{k-1}H^k(\Omega)$ of
    \[
    \min_{\substack{\bu _\theta \in P_{k-1}H^k(\Omega)^3 \\ t_\eta \in P_{k-1}H^k(\Omega)}} \mathcal J^\delta(\boldsymbol \theta, \boldsymbol \eta),
    \]
    we achieve
    \begin{equation*}
        \left\|\hat t_\eta-t^*\right\|_{L^{\infty}\left(\Omega_d\right)} \leq \varepsilon + CB(4\max\{\lambda_{BCN}, \lambda_{DATA}\}\delta^2)^\nu.
    \end{equation*}

    In particular, we get for the number of nodes $N=\mathcal O(h^{-3})$
    \begin{equation}\label{res:eqn:h-scaling}
        N = \Omega \left( \epsilon ^{-12/(\nu(k-2))} \right).
    \end{equation}
\end{theorem}
\begin{proof}
    We again consider the noiseless problem first.
    Assuming the same mesh size $h>0$ for the state and the parameter leads to, using \Cref{met:thm:fem-approx},
    \begin{equation}\label{opt:eqn:approxerrorestimate-mesh}
            \cJ(\mvect \theta,\mvect \eta) \leq
            C \left( \norm{\mathcal P _{k-1} \vect u - \vect u^*}_{H^2(\Omega)}^2 + \norm{\mathcal P _{k-1} t - t^*}_{H^1(\Omega)}^2 \right) \leq \tilde{\varepsilon},
        \end{equation}
    with a given tolerance level $\tilde \varepsilon > 0$, where we have  
        \begin{equation*}
            h = \mathcal O \left( \tilde \epsilon ^{1/(2(k-2))} \right).
        \end{equation*}
        
    We come now to the same conclusion as in \Cref{thm:approx-error-shear-modulus} with the rate
        \begin{equation*}
            h = \mathcal O \left( \epsilon ^{4/((k-2)\nu)} \right).
        \end{equation*}

    Since we can also use the estimate of \cref{eqn:loss-estimate-noisy-noiseless} and subadditivity as done in \Cref{thm:approx-error-shear-modulus}, we get the same decomposition for the noisy problem.
\end{proof}

It is worth noting that the exponent in \cref{res:eqn:h-scaling} appears, at first glance, to be independent of the spatial dimension. 
However, when reformulated in terms of the number of elements, a dimension-dependent behaviour becomes evident.

Moreover, a subtle distinction emerges in the numerical implementation when comparing \Cref{thm:approx-error-shear-modulus} and \Cref{thm:approx-error-shear-modulus-mesh}. 
In the PINNs framework, increased regularity of the ground-truth solution automatically leads to improved convergence rates. 
In contrast, the mesh-based setting requires the projection operator $\mathcal P_{k-1}$ to map onto higher-degree polynomial spaces that are at least $\mathcal C^1$-conforming in order to exploit such regularity. 
As a result, one must assume maximal smoothness of the ground-truth solution a priori, without being able to leverage any additional regularity that may be present.

It may be tempting to apply the same least-squares framework with a FE ansatz space, which falls into the class of least-squares FE methods.
As discussed in detail in \cite{gunzburger_least-squares_2009}, the data space norm plays a crucial role in this context. 
When using $L^2$ norms, the resulting solutions can vary depending on the conformity of the FE space. 
Specifically, $\mathcal C^1$-conforming elements recover the intended solution, whereas elements with lower regularity, which correspond to subspaces of $H^1$, effectively solve a fourth-order problem instead. 
This reveals once more one of the main challenges when using FE in this least-squares setting—namely the construction of highly regular elements.

\section{Numerical Results}
Having established a-posteriori error estimates that account for both approximation error and noise levels, we now demonstrate the reconstruction capabilities of the proposed approach through a simple numerical example. 
Specifically, we consider an in-silico generated forward problem and study how the accuracy of the reconstructed parameter depends on the number of degrees of freedom in the discretisation and the imposed noise level.

\textbf{The forward problem.} The governing constitutive equation is the isotropic model given in \cref{met:eqn:isotropic_law}. The domain is a cuboid defined as  $\bar \Omega = [0, 2] \times [0, 2] \times [0, 1] \si{\milli \metre}$. The boundary conditions for the forward problem are summarised in \Cref{num:tab:bclinear}. The Lamé's first parameter is fixed at $\lambda=\SI{650.0}{\kilo \pascal}$, and the ground-truth shear modulus is defined as
\[
\mu^*(x) = \begin{cases}
    \SI{8}{\kilo \pascal}, & \text{ if } x_1 < x_2, \\
    \SI{16}{\kilo \pascal}, & \text{ for } x_1 \geq x_2.
\end{cases}
\]
The ground-truth displacement data is generated using the Firedrake FE package in Python. 
The code to generate the forward problem as well as the implementation of the various methods for the inverse problem are available under \url{https://github.com/HoeflerMatthias/LinearElasticity.git}. 
We use linear Lagrange elements defined on a uniform cubic mesh with $80 \times 80 \times 40$ elements, leading to a resolution of \SI{0.025}{\milli\metre} in each direction. 
Noise is only considered for the state, resulting in a noisy data source $\bu(\delta)=\bu^*+\sigma \boldsymbol \varepsilon$ where $\varepsilon$ is a standard normally distributed random field with variance given by
$
\sigma = \delta\norm{\bu^*}_\infty/3,
$
where we consider the noise levels $\delta \in \{0.01, 0.05, 0.1\}$.

\begin{table}[ht]
    \centering
    \begin{tabular}{cccl}
        \toprule
        \textbf{face} & \textbf{symbol} & \textbf{type} & \textbf{data} \\ 
        \midrule
        $x=0$  & $\Gamma_{x^-}$ & Dirichlet  & \; $u_1=0$       \\ 
        $x=2$  & $\Gamma_{x^+}$ & Neumann    & \; $\vect p= \vect 0$       \\ 
        $y=0$  & $\Gamma_{y^-}$ & Dirichlet  & \; $u_2=0$       \\ 
        $y=2$  & $\Gamma_{y^+}$ & Neumann    & \; $\vect p= \vect 0$       \\ 
        $z=0$  & $\Gamma_{z^-}$ & Dirichlet  & \; $u_3=0$       \\ 
        $z=1$  & $\Gamma_{z^+}$ & Neumann    & \; $\vect p=-p \vect e_3$       \\ 
        \bottomrule
    \end{tabular}
    \caption{Boundary conditions for the linear elasticity example. The amplitude of the traction vector is given by $p=\SI{10.0}{\kilo \pascal}$.}
    \label{num:tab:bclinear}
\end{table}

\begin{figure}[t]
    \centering
    \includegraphics[width=0.8\linewidth]{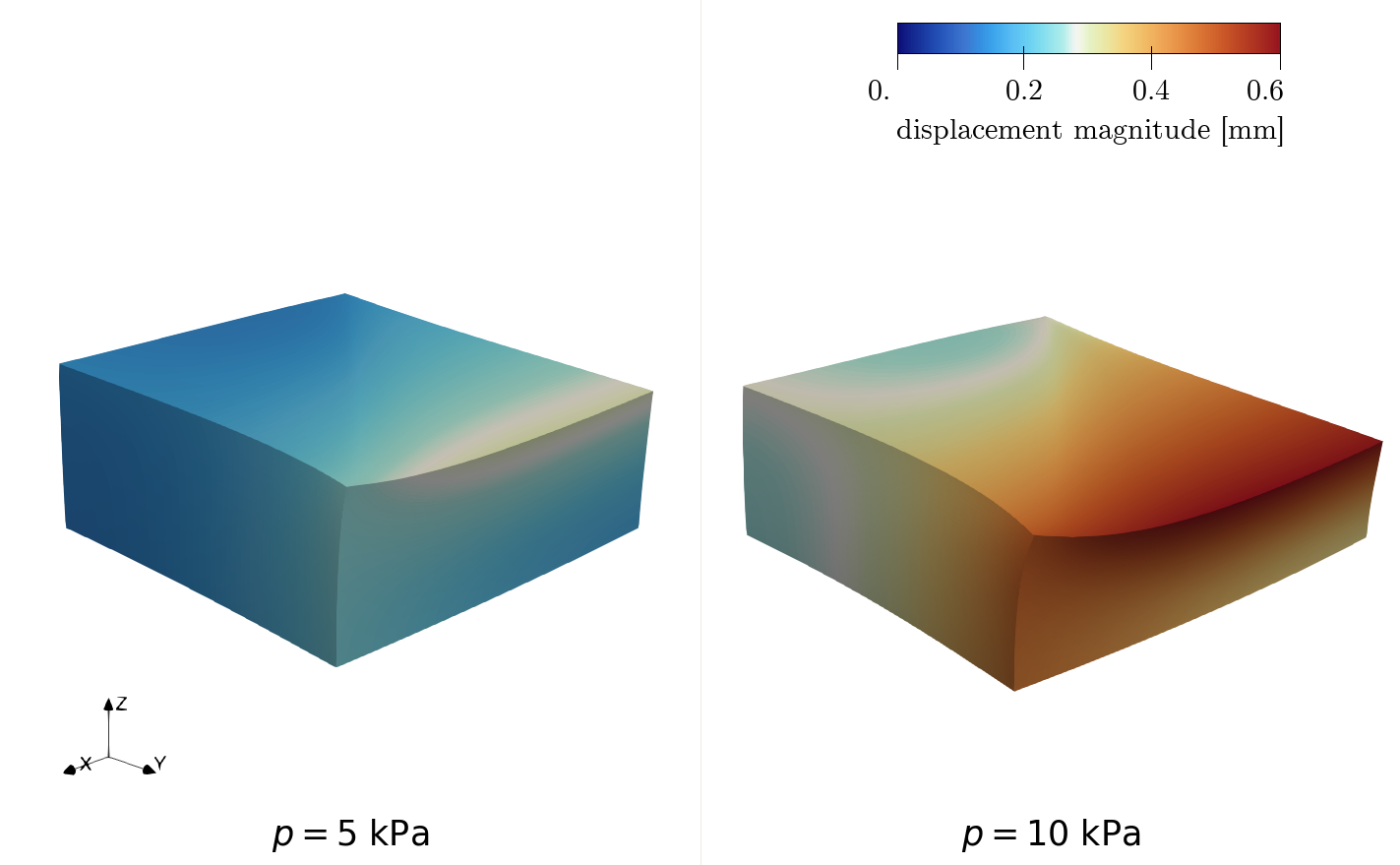}
    \caption{Deformed configuration according to the generated ground-truth displacement $\vect u^*$ with different loads applied on the top surface.}
    \label{fig:placeholder}
\end{figure}

\textbf{The inverse problem.} Assuming access to the (noisy) displacement data and full knowledge of the boundary conditions and Lamé's first constant $\lambda$, the objective is to reconstruct the spatially varying shear modulus $\mu$.

In practical applications, measurement data is often sparse or limited. To model scarcity, we define the observational operator
\[
\mathcal O(\bu(\delta)) = \sum_{i=1}^{N_\text{DATA}} \chi_{\Omega_i}\bu(\delta)_i,
\]
where each $\Omega_i$ denotes a measurable subdomain over which the measurement is supported, and $\chi_{\Omega_i}$ is the corresponding indicator function. 
This operator effectively encodes piecewise constant observations of the field $\bu(\delta)$. 
In the context of PINNs, a common sampling strategy is to set $|\Omega_i|=1/N_{\text{DATA}}$, reflecting uniformly distributed point-wise data. 
In contrast, for FE methods, each region $\Omega_i$ corresponds to the support of an individual element. 
Thus, it naturally inherits its geometric structure.

To quantify accuracy, we define two error metrics. The state error is measured relative to the ground-truth solution as
\[
\epsilon_{\vect u; \text{rel}} = \frac{\norm{\vect u - \vect u^*}_{L^2(\Omega,\R^3)}}{\norm{\vect u^*}_{L^2(\Omega,\R^3)}},
\]
where the spatial discretisation is according to the ground-truth mesh. For the parameter, we define a notion of error via
\[
\epsilon_{\mu; \text{rel}} = \frac{1}{|\Omega|^{1/2}}\norm{\frac{\mu - \mu^*}{\mu^*}}_{L^2(\Omega)},
\]
which is pointwise in the sense that we compute the relative error separately for each point in the domain before integrating over it. 
This contrasts with defining a relative error using two separate integrals for the numerator and denominator.
Note that the factor $|\Omega|^{-1/2}$ allows to interpret $\epsilon_{\mu; \text{rel}}$ as a relative RMSE, hence it is suitable to direct approximation with a sample mean w.r.t. a uniform distribution over the mesh.

\begin{remark}
    While this notion of pointwise, relative error measured in the $L^2(\Omega)$ norm is not directly related to the $L^\infty(\Omega_d)$ bounds derived so far, it is of highly practical interest since it also gives a notion on the pointwise amplification. 
    Moreover, since $\min_{x \in \bar \Omega} |\mu^*(x)|>0$, it can be easily upper-bounded by the $L^\infty(\Omega)$ norm of the absolute deviation from above. 
\end{remark}

\subsection{PINNs approach}\label{sec:pinn_num}
We begin our investigation of the inverse problem by employing neural network-based ansatz spaces. 
One of the primary challenges in PINNs lies in minimising the optimisation error. 
This issue becomes particularly significant in the context of inverse problems, where the selection of appropriate loss weights for each residual term greatly impacts parameter identification—especially in the presence of noise~\cite{hofler_physics-informed_2025}. 
Both increasing network size and noise level can lead to an increase in optimisation error due to the resulting complexity of the loss landscape. 

To mitigate the influence of optimisation error on our numerical results, we incorporate several standard modifications to the vanilla PINN training algorithm \cite{anagnostopoulos_residual-based_2024, tancik_fourier_2020, sukumar_exact_2022}. 
For a comprehensive discussion on the training procedure, we refer the reader to \cite{hofler_physics-informed_2025}. 
Here, we briefly highlight the key elements. 
The optimisation procedure follows a 2x2 step scheme: In a first data-only phase we fit the available displacement data to the network. 
In a second phase, we use both the data and the physics regulariser. 
Each phase consists of a fixed number of iterations using the Adam optimiser, followed by a second stage employing the BFGS optimiser. 
We emphasise the use of boundary layers \cite{sukumar_exact_2022} to impose Dirichlet conditions, thereby yielding ansatz spaces similar in structure to those derived from FE spaces.

For each of the two networks $\bu_\theta$ and $\mu_\eta$, we employ tanh activation functions and three hidden layers with equal widths $N_{\bu}$ and $N_\mu$, respectively. We assume scarce measurements of the displacement field $\bu(\delta)$, which are given by $N_{\text{DATA}}=\num{1000}$ randomly sampled points in the domain.

\Cref{tab:pinnresults} lists results for each combination of network size and noise level. 
Each of the entries denotes an average over five different seeds to take into account different network initialisations.
Interestingly, the limited improvement observed with increasing $N_\mu$ suggests that the spatial variability of the shear modulus can be effectively captured by relatively simple models. 
As illustrated in \Cref{fig:pinnresults}, increasing the network size for the shear modulus can, in some cases—particularly under higher noise—lead to a slight degradation in the reconstruction quality. 
This trend indicates not only a rise in optimisation error but also a decrease of the regularising effect of the discretisation. 
Furthermore, model accuracy is more strongly influenced by the quality of the displacement field approximation than by the expressiveness of the parameter field network.

In fact, the displacement network size has a marked effect on the relative error in shear modulus estimation: The error decreases from approximately $\SI{20}{\percent}$ for small networks ($N_\bu=4$) to around $\SI{8}{\percent}$ for the largest tested network ($N_\bu=32$). 
Moreover, the impact of noise becomes more pronounced for larger parameter networks than for smaller ones, as also seen in \Cref{fig:pinnresults}. 
{
\sisetup{
  round-mode = places,
  round-precision = 3,
  round-pad = false,          
  group-digits = false, 
  output-decimal-marker = {.}
}
\begin{table}[t]
    \centering
    \begin{tabular}{ccc|cc||ccc|cc}
         $ \mathbf N_{\mathbf u}$ & $\mathbf N_{\mu}$ & \textbf{noise level} & $\boldsymbol \epsilon_{\vect u; \textbf{rel}}$ &  $\boldsymbol \epsilon_{\mu; \textbf{rel}}$ &
         $ \mathbf N_{\mathbf u}$ & $\mathbf N_{\mu}$ & \textbf{noise level} & $\boldsymbol \epsilon_{\vect u; \textbf{rel}}$ &  $\boldsymbol \epsilon_{\mu; \textbf{rel}}$
         \\ \midrule
         \num{4} & \num{2} & \SI{1}{\percent} & \num{0.02502} & \num{0.20579}      & \num{16} & \num{2} & \SI{1}{\percent} & \num{0.00447} & \num{0.10838} \\
          &  & \SI{5}{\percent} & \num{0.02534} & \num{0.19280}      &  &  & \SI{5}{\percent} & \num{0.00624} & \num{0.10888} \\
          &  & \SI{10}{\percent} & \num{0.03029} & \num{0.20099}      &  &  & \SI{10}{\percent} & \num{0.01005} & \num{0.11295} \\
          & \num{4} & \SI{1}{\percent} & \num{0.02742} & \num{0.24506}      &  & \num{4} & \SI{1}{\percent} & \num{0.00462} & \num{0.11120} \\
          &  & \SI{5}{\percent} & \num{0.02853} & \num{0.21891}      &  &  & \SI{5}{\percent} & \num{0.00619} & \num{0.11133} \\
          &  & \SI{10}{\percent} & \num{0.02621} & \num{0.18961}      &  &  & \SI{10}{\percent} & \num{0.01046} & \num{0.12211} \\
          & \num{8} & \SI{1}{\percent} & \num{0.02702} & \num{0.26634}      &  & \num{8} & \SI{1}{\percent} & \num{0.00402} & \num{0.11114} \\
          &  & \SI{5}{\percent} & \num{0.02611} & \num{0.20100}      &  &  & \SI{5}{\percent} & \num{0.00631} & \num{0.11354} \\
          &  & \SI{10}{\percent} & \num{0.02725} & \num{0.19577}      &  &  & \SI{10}{\percent} & \num{0.01069} & \num{0.12618} \\
          & \num{16} & \SI{1}{\percent} & \num{0.02516} & \num{0.18831}      &  & \num{16} & \SI{1}{\percent} & \num{0.00401} & \num{0.11040} \\
          &  & \SI{5}{\percent} & \num{0.02240} & \num{0.18333}      &  &  & \SI{5}{\percent} & \num{0.00627} & \num{0.11535} \\
          &  & \SI{10}{\percent} & \num{0.02574} & \num{0.18662}      &  &  & \SI{10}{\percent} & \num{0.01084} & \num{0.13177} \\
         \num{8} & \num{2} & \SI{1}{\percent} & \num{0.00907} & \num{0.13089}      & \num{32} & \num{2} & \SI{1}{\percent} & \num{0.00190} & \num{0.08260} \\
          &  & \SI{5}{\percent} & \num{0.00989} & \num{0.13314}      &  &  & \SI{5}{\percent} & \num{0.00649} & \num{0.09186} \\
          &  & \SI{10}{\percent} & \num{0.01255} & \num{0.13796}      &  &  & \SI{10}{\percent} & \num{0.01358} & \num{0.11042} \\
          & \num{4} & \SI{1}{\percent} & \num{0.00873} & \num{0.13303}      &  & \num{4} & \SI{1}{\percent} & \num{0.00183} & \num{0.08298} \\
          &  & \SI{5}{\percent} & \num{0.00956} & \num{0.13341}      &  &  & \SI{5}{\percent} & \num{0.00672} & \num{0.09861} \\
          &  & \SI{10}{\percent} & \num{0.01241} & \num{0.13942}      &  &  & \SI{10}{\percent} & \num{0.01192} & \num{0.11887} \\
          & \num{8} & \SI{1}{\percent} & \num{0.00853} & \num{0.13041}      &  & \num{8} & \SI{1}{\percent} & \num{0.00186} & \num{0.08345} \\
          &  & \SI{5}{\percent} & \num{0.00933} & \num{0.13254}      &  &  & \SI{5}{\percent} & \num{0.00717} & \num{0.10651} \\
          &  & \SI{10}{\percent} & \num{0.01226} & \num{0.13985}      &  &  & \SI{10}{\percent} & \num{0.01066} & \num{0.12042} \\
          & \num{16} & \SI{1}{\percent} & \num{0.00877} & \num{0.13281}      &  & \num{16} & \SI{1}{\percent} & \num{0.00185} & \num{0.08486} \\
          &  & \SI{5}{\percent} & \num{0.00902} & \num{0.12952}      &  &  & \SI{5}{\percent} & \num{0.00736} & \num{0.11188} \\
          &  & \SI{10}{\percent} & \num{0.01187} & \num{0.13864}      &  &  & \SI{10}{\percent} & \num{0.01087} & \num{0.12716} \\
        \hline
    \end{tabular}
    \caption{Summary of results for the PINNs-based approach. The table shows, for particular network sizes $N_{\bu}$ and $N_\mu$ and noise level $\delta$, the mean reconstruction error $\epsilon_{\mu; \text{rel}}$ through five different initialisations of the algorithm.}
    \label{tab:pinnresults}
\end{table}
}
\begin{figure}
    \centering
    \includegraphics[width=\linewidth]{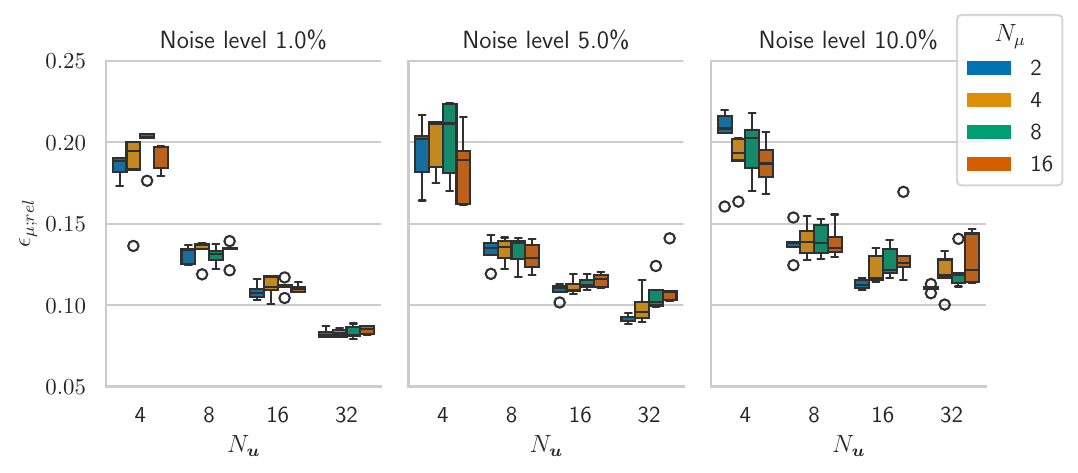}
    \caption{Box-plots of the relative error on $\epsilon_{\mu; \text{rel}}$ for each configuration. Each plot shows the results for a different noise level $\delta$. The x-axis indicates the width of the displacement network $N_\bu$ and the color the width of the shear modulus network $N_\mu$.}
    \label{fig:pinnresults}
\end{figure}

\subsection{Comparison with FE approaches}
As previously noted, constructing a consistent least-squares discretisation within the FE framework is non-trivial. 
Nevertheless, we can still aim to solve the inverse problem with a linear, localised ansatz space. We start by presenting a strategy to tackle the inverse problem in an all-at-once approach using FE ansatz spaces for the least-squares problem.
For the purpose of a sound comparison method, we also consider a solution strategy based on the weak formulation introduced in \cite{pozzi_reconstruction_2024}. However, in this work we compare not only against the proposed reduced formulation but also against an all-at-once variant of the same method.

In all approaches, we discretise the domain using a mesh with $10N \times 10N \times 5N$ elements, where $N \in \{1, 2, 4, 6\}$ serves as a proxy for the parameter count to compare different mesh sizes against one another, similar to $N_{\vect u}$ and $N_\mu$ in \Cref{sec:pinn_num}.

For comparison, we project the obtained displacement and parameter fields back to the finer mesh of the ground-truth solution.

Moreover, similar to \cite{pozzi_reconstruction_2024}, the shear modulus is reparametrised as $\mu(x) = \SI{8}{\kilo\pascal} \cdot \alpha(x)$ and we also consider additional regularisation on the parameter, denoted by $\mathcal{R}(\alpha)$. 
For the choices of regularisation, we consider three variants: $L^2$ norm $\mathcal R(\alpha)=\norm{\alpha}_{L^2(\Omega)}^2$, $H^1$ semi-norm $\mathcal R(\alpha)=\norm{\nabla \alpha}_{L^2(\Omega)}^2$, or a smoothed approximation of total variation (TV) $\mathcal R(\alpha)=\int_\Omega \sqrt{|\nabla \alpha|^2+\varepsilon}\;\mathrm dx$, as discussed in \cite{bredies_higher-order_2020}.

\subsubsection{Least-squares FE approach}
We start with the analogous approach to the PINNs framework. However, we use FE ansatz spaces instead. 
As noted in \cite{gunzburger_least-squares_2009}, solving the least-squares formulation in the $L^2$ norm over subspaces of $H^1$ does not lead to the desired solution. 
The key idea is to have a FE representation $\mathbf u_\theta \in H^1(\Omega)$ which we want to lift to a quasi-$\mathcal C^1$-conforming representation via regularisation. 
In particular, if we use Lagrange elements of order $r$, abbreviated as Pr, we get a continuous representation with piecewise continuous gradient. Hence, it comes natural to penalise the jump of the gradient between interior facets. 
More precisely, denote by $\mathcal F_h^{\text{int}}$ the set of all interior facets, and for a given facet $F\in \mathcal F_h^{\text{int}}$ its two neighbouring elements with $\Omega_i^+, \Omega_i^- \subset \Omega$, and with $\mathrm d s_F$ the corresponding facet measure. 
We define the jump at $F$ of a sufficiently smooth function $v$ as $[v]_F=v|_{\Omega_i^+}-v|_{\Omega_i^-}$. 
The proposed regulariser is defined as
\[
\mathcal R_{\mathcal C^1}(\bu) = \sum_{F \in \mathcal F_h^{\text{int}}} \int_F \norm{[\nabla \bu]_F}_2^2 \mathrm d s_F.
\]

Still, tackling the optimisation problem in an all-at-once manner remains difficult, both for the least-squares as well as for the weak formulation described in \Cref{sec:weak}. 
Let
\begin{equation}\label{eqn:reg-c1}
\mathcal L(\bu, \alpha) =  \mathcal J^\delta(\bu, \alpha) + 
\lambda \mathcal R(\alpha)
+ \lambda_{\mathcal C^1} \mathcal R_{\mathcal C^1}(\bu)
\end{equation}
be the modified objective, where the shear modulus is reparametrised as $\mu(x) = \SI{8}{\kilo\pascal} \cdot \alpha(x)$. We carry out optimisation with a block coordinate descent strategy \cite{doi:10.1137/100802001}. Namely, for given iterates $\bu^k, \alpha^k$, we solve
\[
\begin{cases}
    \bu^{k+1} \in \arg \min_{\bu} \mathcal L(\bu, \alpha^k), \\
    \alpha^{k+1} \in \arg \min_{\alpha} \mathcal L(\bu^{k+1}, \alpha).
\end{cases}
\]
The individual optimisation problems are solved with BFGS.

It is important to emphasise that achieving a reliable local minimum remains numerically challenging. In fact, the FE-based solution strategy exhibits even greater sensitivity to the choice of weighting parameters than the PINNs-based approach. 
Successful identification was only attainable under full-field measurements, i.e., when each finite element is associated with at least one observation. 
Notably, although the regularisation term in \cref{eqn:reg-c1} imposes only a soft constraint and does not enforce strict $\mathcal C^1$-conformity, its inclusion proved essential: in its absence, the identification process consistently failed. 
However, it was still only successful in cases with the smallest noise level $\delta=0.01$ and coarse grids $N \in \{1,2\}$. 
In more detail, for $N=1$, $\epsilon_{\mu;\text{rel}}=\num[round-mode = places, round-precision = 3]{0.157895}$, whereas for $N=2$, $\epsilon_{\mu;\text{rel}}=\num[round-mode = places, round-precision = 3]{0.172409}$. 
For every other setup, the reconstruction was considered as not successful in the sense that the relative error on $\mu$ was above $\SI{30}{\percent}$.

\begin{remark}
    An alternative route toward least-squares optimisation is proposed in \cite{karnakov_solving_2023}, where, rather than employing a function space approximation, the authors adopt an operator discretisation of the PDE. 
    Specifically, they use a finite difference discretisation of both the governing equations and boundary conditions, which naturally results in a matrix-based representation of the state and parameter fields. 
    Their framework, termed Optimizing a Discrete Loss (ODIL), is applied to a range of benchmark problems.

    In a recent extension, ODIL has been employed in medical imaging applications~\cite{balcerak_physics-regularized_2024}, where multi-modal data is combined with coupled physical models, including the equations of quasi-static elasticity. 
    While this approach demonstrates considerable flexibility in discrete optimisation, it does not align with our function space framework. 
    In particular, from the perspective of function space approximations, finite difference methods lack the $\mathcal{C}^1$-conformity required for the least-squares fromulation considered here. 
    As such, their direct application in our setting remains limited.
\end{remark}

\subsubsection{Weak formulation}\label{sec:weak}
To establish a solid baseline for comparison, we adopt a weak formulation of the inverse problem. 
In particular, we consider the reduced approach presented in \cite{pozzi_reconstruction_2024}, together with its reformulation as an all-at-once approach. 
This methodology is especially relevant to our setting, as it targets the same goal: the identification of spatially varying material properties in cardiac tissue.

We begin with a brief summary of their approach. The problem to solve is given by
\[
\min _{\boldsymbol{u}, \alpha} 
\frac{1}{2}\norm{\bu-\bu(\delta)}_{L^2(\Omega)}^2 + \lambda \mathcal R(\alpha), \quad \text { s.t. } \quad\left\{\begin{aligned}
-\nabla \cdot \mathbb C_\mu \symgrad \bu =0, & \text { in } \Omega, \\
\mat B\bu=0, & \text { on } \Gamma_D, \\
\mathbb C_\mu \symgrad \bu \cdot \mathbf {n}=\mathbf p, & \text { on } \Gamma_N, \\
\alpha \geq 0, & \text { in } \Omega .
\end{aligned}\right.
\]
where $\mat{B}$ denotes a selection matrix for the relevant displacement components, as defined in \Cref{num:tab:bclinear}.

Constraints are enforced via a Lagrange multiplier field $\mathbf{m}$, leading to the Lagrangian
\[
\mathcal{L}(\bu, \alpha, \mathbf{m})=
\frac{1}{2}\norm{\bu-\bu(\delta)}_{L^2(\Omega)}^2 + \lambda \mathcal R(\alpha)
+\int_{\Omega} \nabla \cdot \mathbb C_\mu \symgrad \bu: \nabla \mathbf{m} \; \mathrm{d} \mathbf{x}
-\int_{\Gamma_N} \mathbf p \cdot \mathbf m \;\mathrm{d} \mathbf{s}.
\]
The all-at-once approach seeks a stationary point of $\mathcal{L}$ by solving the optimality condition
\begin{equation}\label{eqn:aao_kkt}
\delta \mathcal{L}(\bu, \alpha, \mathbf{m}) = 0,
\end{equation}
which corresponds to solving the full set of Karush-Kuhn-Tucker (KKT) conditions, involving the primal, adjoint, and control equations, as detailed in \cite{pozzi_reconstruction_2024}. 
However, it is well known that solving \cref{eqn:aao_kkt} is numerically demanding \cite{haber2001preconditioned}, due to ill-conditioning and the risk of convergence to local minima.

We initialise the parameter field with a constant value $\alpha(x) = 1.5$, representing the average of two target regions. 
This initial guess is used to compute the forward solution $\bu$, which then serves as a starting point for the optimisation. 
The optimisation is carried out using a Newton method with line search, as implemented in the PETSc toolkit.

From the optimality condition \cref{eqn:aao_kkt}, one can also derive the parameter-to-solution map $\alpha \mapsto \bu$, leading to a reduced formulation of the problem, in which optimisation is performed over $\alpha$ only (see \cite{pozzi_reconstruction_2024} for derivation details). The same initialisation $\alpha(x) = 1.5$ is used here as well.

Since the variational methods exhibit sufficient stability, we apply the same sparse data setting as in the PINN approach, using $N_{\text{DATA}}=\num{1000}$ randomly chosen element locations.
{
\sisetup{
  round-mode = places,
  round-precision = 3,
  round-pad = false, 
  group-digits = false,
  output-decimal-marker = {.}
}
\begin{table}[t]
        \centering
        \begin{tabular}{cc|c|cc|c}
             $\mathbf N$ & \textbf{noise level} & \textbf{regularisation} & $\boldsymbol \epsilon_{\vect u; \textbf{rel}}$ &  $\boldsymbol \epsilon_{\mu; \textbf{rel}}$ & \textbf{no. iter} \\ \hline
             \num{1} & \SI{1}{\percent} & H1, {\num{5e-03}} & \num{0.110552} & \num{0.221435} & \num{4} \\
                 & \SI{5}{\percent} & L2, {\num{5e-03}} & \num{0.131032} & \num{0.220454} & \num{5} \\
                 & \SI{10}{\percent} & H1, {\num{5e-03}} & \num{0.114438} & \num{0.222031} & \num{4} \\
  \hline
         \num{2} & \SI{1}{\percent} & H1, {\num{1e-06}} & \num{0.010257} & \num{0.131662} & \num{8} \\
                 & \SI{5}{\percent} & H1, {\num{5e-06}} & \num{0.012441} & \num{0.138193} & \num{5} \\
                 & \SI{10}{\percent} & H1, {\num{1e-05}} & \num{0.013970} & \num{0.145347} & \num{4} \\
  \hline
         \num{4} & \SI{1}{\percent} & H1, {\num{1e-05}} & \num{0.007779} & \num{0.125957} & \num{4} \\
                 & \SI{5}{\percent} & H1, {\num{1e-05}} & \num{0.007758} & \num{0.125444} & \num{4} \\
                 & \SI{10}{\percent} & H1, {\num{1e-05}} & \num{0.008592} & \num{0.129754} & \num{4} \\
  \hline
         \num{6} & \SI{1}{\percent} & H1, {\num{1e-05}} & \num{0.006427} & \num{0.124112} & \num{4} \\
                 & \SI{5}{\percent} & H1, {\num{1e-05}} & \num{0.007078} & \num{0.124148} & \num{4} \\
                 & \SI{10}{\percent} & H1, {\num{1e-05}} & \num{0.008606} & \num{0.133016} & \num{4} \\
  \hline
        \end{tabular}
        \caption{Summary of results for the weak formulation in an all-at-once approach. The table shows, for a particular mesh size $N$ and noise level $\delta$, the best result for the reconstruction error $\epsilon_{\mu; \text{rel}}$ obtained through all different choices of regularisation. The last column indicates the number of BFGS iterations performed until the specified tolerance was achieved.}
        \label{tab:res-mesh-based-kkt}
    \end{table}

\begin{table}[t]
        \centering
        \begin{tabular}{cc|c|cc|c}
             $\mathbf N$ & \textbf{noise level} & \textbf{regularisation} & $\boldsymbol \epsilon_{\vect u; \textbf{rel}}$ &  $\boldsymbol \epsilon_{\mu; \textbf{rel}}$ & \textbf{no. iter} \\ \hline
             \num{1} & \SI{1}{\percent} & TV, {\num{5e-04}} & \num{0.096807} & \num{0.177228} & \num{146} \\
                 & \SI{5}{\percent} & TV, {\num{1e-04}} & \num{0.068991} & \num{0.172800} & \num{108} \\
                 & \SI{10}{\percent} & TV, {\num{5e-04}} & \num{0.096237} & \num{0.181044} & \num{142} \\
  \hline
         \num{2} & \SI{1}{\percent} & TV, {\num{5e-05}} & \num{0.028530} & \num{0.104526} & \num{207} \\
                 & \SI{5}{\percent} & TV, {\num{1e-05}} & \num{0.024554} & \num{0.110371} & \num{140} \\
                 & \SI{10}{\percent} & TV, {\num{5e-05}} & \num{0.028708} & \num{0.108724} & \num{201} \\
  \hline
         \num{4} & \SI{1}{\percent} & TV, {\num{5e-06}} & \num{0.007935} & \num{0.062640} & \num{254} \\
                 & \SI{5}{\percent} & TV, {\num{1e-05}} & \num{0.011829} & \num{0.080334} & \num{264} \\
                 & \SI{10}{\percent} & TV, {\num{5e-05}} & \num{0.016137} & \num{0.099767} & \num{404} \\
  \hline
         \num{6} & \SI{1}{\percent} & TV, {\num{5e-06}} & \num{0.004798} & \num{0.061650} & \num{443} \\
                 & \SI{5}{\percent} & TV, {\num{1e-05}} & \num{0.008422} & \num{0.069610} & \num{394} \\
                 & \SI{10}{\percent} & TV, {\num{5e-05}} & \num{0.017978} & \num{0.092760} & \num{525} \\
  \hline
        \end{tabular}
        \caption{Summary of results for the weak formulation in a reduced approach. The table shows, for a particular mesh size $N$ and noise level $\delta$, the best result for the reconstruction error $\epsilon_{\mu; \text{rel}}$ obtained through all different choices of regularisation. The last column displays the number of BFGS iterations carried out until the tolerance was reached.}
        \label{tab:res-mesh-based-red}
    \end{table}
}
A comparison between Tables \ref{tab:res-mesh-based-kkt} and \ref{tab:res-mesh-based-red} reveals that the reduced formulation consistently yields superior reconstruction accuracy for the parameter $\mu$, and comparable results for the state $\vect u$. 
The primary advantage of the all-at-once approach lies in its computational speed; however, this comes at the expense of solution stability. 
In practice, most choices of regularisation type and parametrisation for the all-at-once method result in instabilities, with only a few successful trials observed. 
Notably, while TV regularisation consistently outperforms other regularisation strategies in the reduced setting, it proved insufficiently robust for the all-at-once formulation in the majority of cases. 
This instability further highlights the ill-conditioning inherent to the all-at-once approach.

\Cref{fig:fieldcomp} presents a visual comparison of the best reconstructions achieved by the various methods. The PINNs approach yields a straight boundary, benefiting from a well-designed parameter field architecture, though it introduces slight blurring. Comparable reconstruction quality is attained only by the reduced weak formulation, which provides a sharper separation between regions due to TV regularisation, albeit with noticeable artifacts at the corners. The all-at-once weak formulation produces a blurrier yet still reasonable result, shaped primarily by the $H^1$ regularisation. In principle, TV could be beneficial, however, is not stable enough to be applied successfully. The least-squares FE method produces the weakest result. While it allows rough identification of the regions, it fails to capture the simplicity of the ground-truth parameter field.

\section{Conclusion}
We have established discretisation-invariant a-posteriori error estimates for a broad class of parameter identification problems in three-dimensional elasticity, with a particular focus on observational noise and approximation error. 
The theoretical findings also offer valuable guidance for the design of all-at-once optimisation schemes tailored to such inverse problems. 
In particular, they provide a rigorous justification for the use of neural network-based ansatz spaces within this setting.

Compared to FE methods, PINNs exhibit several methodological advantages, particularly when the target solution is smooth. Unlike classical FE approaches, which rely on explicitly constructed function spaces to enforce regularity, PINNs implicitly encode smoothness through their neural network architecture. 
This becomes especially advantageous in higher spatial dimensions, where constructing high-order FE spaces is technically challenging.

Moreover, we identified persistent challenges in extending analogous strategies to FE-based ansatz spaces. These limitations were not only analysed from a theoretical standpoint but also confirmed through numerical experiments. 
To contextualise our findings, we conducted a detailed numerical comparison with solution approaches based on weak formulations. The results demonstrated that the reduced formulation is more robust in mitigating numerical issues related to optimisation error than the all-at-once strategy.

From a computational standpoint, the all-at-once approach in the weak formulation appears to be the most efficient. However, this performance gain comes at the cost of increased instability and a limited set of viable regularisation strategies. 
Despite these challenges, the method holds promise for future developments, particularly as a pre-processing stage that could provide improved initialisations for more robust schemes.

The all-at-once formulation using PINNs exhibits properties comparable to those of the reduced weak formulation. 
While the latter yields slightly higher reconstruction accuracy, the former demonstrates superior computational efficiency—especially beneficial when high-resolution solutions are required.

In summary, future advancements may benefit from hybrid strategies that integrate the strengths of different formulations. 
Although a stable FE-based least-squares method remains elusive, reducing the nonlinearity in the optimisation problem within PINNs may help mitigate optimisation errors. 
Nevertheless, it remains an open question how to best combine PINNs with reduced formulations or linear ansatz spaces to fully leverage their respective advantages.

\begin{figure}[ht]
    \centering
    \includegraphics[width=0.8\linewidth]{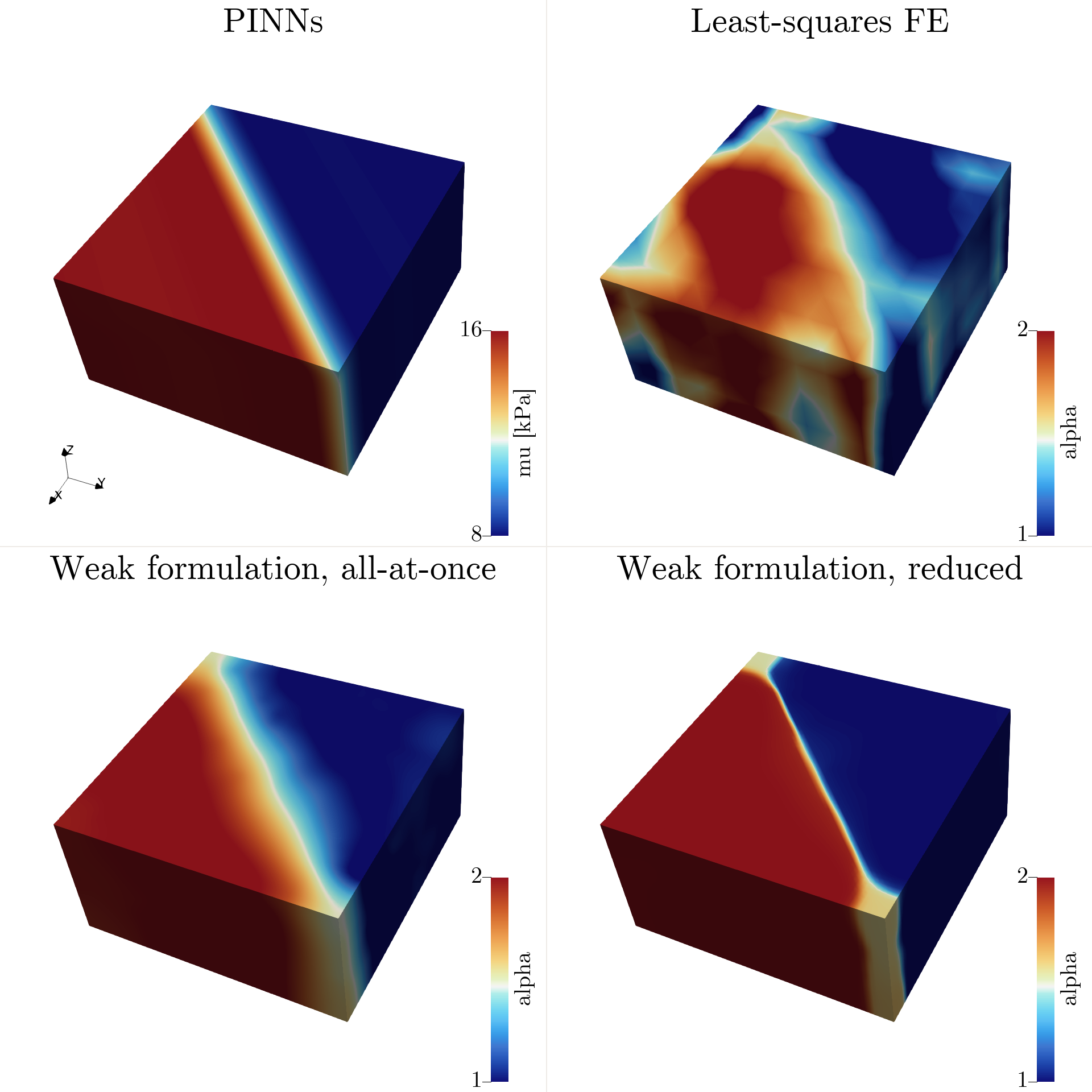}
    \caption{Comparison of the best reconstructed fields for the case $\delta=0.01$. The top left plot displays the PINNs reconstruction with $N_{\mathbf{u}} = 32$ and $N_\mu = 2$, with the color range clipped between $\SIrange{8}{16}{\kilo\pascal}$. The top right plot shows the least-squares FE approach with $N = 1$. The bottom left plot presents the weak formulation using the all-at-once approach for $N = 6$, while the bottom right plot illustrates the weak formulation using the reduced approach for $N = 6$. For the latter three plots, $\alpha$ is clipped between 1 and 2.}
    \label{fig:fieldcomp}
\end{figure}

\section{Acknowledgements}
The authors gratefully acknowledge the assistance of Francesco Regazzoni and Stefano Pagani for their support in the implementation of the PINNs algorithm, in particular for the use of their framework nisaba.
FEM simulations for this study were performed on the Vienna Scientific Cluster (VSC-5) under PRACE project \#72420, which is maintained by the ASC Research Center in collaboration with the Information Technology Solutions of TU Wien. 
FC is member of the INdAM research group GNCS.
FC acknowledges support from the Land Steiermark UFO Grant No. 3026.
This research was funded in whole or in part by the Austrian Science Fund (FWF) 10.55776/F100800.

\nocite{romer2025reduced}

\appendix

\section{Properties of the Loss Functional}

\begin{lemma}\label{opt:lem:optbound}
    For weights $\lambda_1, \lambda_2, \lambda_3 > 0$, and a tolerance level $\varepsilon >0$, the optimal value of the problem
    \begin{equation}\label{opt:eqn:optproblemsurrogate}
        \begin{cases}
            \max \varepsilon_1 + \varepsilon_2 + \varepsilon_3^{1/4} \\
            \text{subject to } \varepsilon_1, \varepsilon_2, \varepsilon_3 \geq 0 \text{ and } \lambda_1 \varepsilon_1^2 + \lambda_2 \varepsilon_2^2 + \lambda_3 \varepsilon_3^2 \leq \varepsilon
        \end{cases}
    \end{equation}
    is bounded by
    \[
    B(\varepsilon) \coloneqq \begin{cases}
        (\varepsilon/\lambda)^{1/8}, & \text{ for } \varepsilon/\lambda \leq (\varepsilon^*)^2, \\
        (\varepsilon^*)^{1/4}+\sqrt{\varepsilon/\lambda-(\varepsilon^*)^2}, & \text{ for } \varepsilon/\lambda > (\varepsilon^*)^2, \\
    \end{cases}
    \]
    where $0<\epsilon^*<1/2$ is a fixed constant and $\lambda = \displaystyle\min_{i=1,2,3} \lambda_i$.
\end{lemma}
\begin{proof}
    We look at the simpler problem
    \begin{equation}\label{opt:eqn:optproblemsimple}
        \begin{cases}
            \max \varepsilon_1 + \varepsilon_2 + \varepsilon_3^{1/4} \\
            \text{subject to } \varepsilon_1, \varepsilon_2, \varepsilon_3 \geq 0 \text{ and } \varepsilon_1^2 + \varepsilon_2^2 + \varepsilon_3^2 \leq \frac{\varepsilon}{\lambda}.
        \end{cases}
    \end{equation}
    Since we have the inclusion
    \begin{equation*}
        \{ \varepsilon_1, \varepsilon_2, \varepsilon_3 \geq 0 \text{ and } \lambda_1 \varepsilon_1^2 + \lambda_2 \varepsilon_2^2 + \lambda_3 \varepsilon_3^2 \leq \varepsilon\}
        \subset \left\{ \varepsilon_1, \varepsilon_2, \varepsilon_3 \geq 0 \text{ and } \varepsilon_1^2 + \varepsilon_2^2 + \varepsilon_3^2 \leq \frac{\varepsilon}{\lambda} \right\},
    \end{equation*}
    the solution to (\ref{opt:eqn:optproblemsimple}) provides an upper bound for (\ref{opt:eqn:optproblemsurrogate}).

    Denote by $0<\varepsilon^*<1/2$ the point such that
    \[
    \frac{d}{d \varepsilon_3} \restr{\left( \varepsilon_3^{1/4}-\varepsilon_3\right)}{\varepsilon^*}=0.
    \]
    After this point, the growth of $\varepsilon_3^{1/4}$ is slower than linear. 
    Hence, for the case $\sqrt{\varepsilon/\lambda} \leq \varepsilon^*$, setting $\varepsilon_3=\sqrt{\varepsilon/\lambda}$ leads to the optimal value of problem \cref{opt:eqn:optproblemsimple}, given by $(\varepsilon/\lambda)^{1/8}$. 
    Otherwise, for the case $\sqrt{\varepsilon/\lambda} > \varepsilon^*$, we set $\varepsilon_3=\varepsilon^*$ and use for the remaining part the linear growth of either $\varepsilon_1$ or $\varepsilon_2$, giving an optimal value of $\varepsilon^*+\sqrt{\varepsilon/\lambda-(\varepsilon^*)^2}$.
    In total, we get 
    \[
    B(\varepsilon) = \begin{cases}
        (\varepsilon/\lambda)^{1/8}, & \text{ for } \varepsilon/\lambda \leq (\varepsilon^*)^2, \\
        (\varepsilon^*)^{1/4}+\sqrt{\varepsilon/\lambda-(\varepsilon^*)^2}, & \text{ for } \varepsilon/\lambda > (\varepsilon^*)^2. \\
    \end{cases}
    \]
\end{proof}

\section{Continuity of the constitutive law}\label{con:sec}

In this section, we examine and verify key properties of the constitutive laws introduced in \cref{met:eqn:isotropic_law} and \cref{met:eqn:anisotropic_law}. 
We begin by deriving rate-independent convergence properties, followed by the formulation of a continuity bound for the corresponding elliptic operators.

\begin{lemma}[Isotropic case]\label{con:lem:iso-continuity}
    If $\vect u_n \rightharpoonup \vect u$ in $H^2(\Omega, \R^3)$ and $\mu_n \rightharpoonup \mu$ in $H^2(\Omega)$, then
    \[
    \mathbb C _ {\mu_n} \nabla \vect u_n \rightharpoonup \mathbb C_\mu \nabla \vect u \quad \text{ in } H(\operatorname{div}; \Omega, \R^3).
    \]
\end{lemma}
\begin{proof}
    We look at each of the two terms of $\mathbb C_\mu \nabla \vect u = \lambda \trace ( \symgrad \vect u ) \mat I + 2 \mu \symgrad \vect u$ separately.
    
    For the first term, the $H^2(\Omega)$ weak convergence of $(\vect u_n)_n$ implies that
    \[
    \divergence \left( \lambda \trace ( \symgrad \vect u_n ) \mat I \right) \rightharpoonup
    \divergence \left( \lambda \trace ( \symgrad \hat {\vect u} ) \mat I \right)
    \quad \text{ in } L^2(\Omega).
    \]

    For the second term, we note that, due to the continuous embeddings $H^2(\Omega) \hookrightarrow W^{1,p}(\Omega)$ for $1 \leq p \leq 6$ and $H^2(\Omega) \hookrightarrow L^\infty(\Omega)$ (see e.g. \cite[Section 10.9]{alt_linear_2016}), we get $\mu_n \nabla \vect u_n \in H^1(\Omega, \R^3)$. By exploiting the fact that, due to Rellich's embedding theorem \cite[Theorem 2.4]{alt_linear_2016}, $\mu_n \longrightarrow \mu$ in $H^1(\Omega)$, we conclude together with the weak convergence $\vect u_n \rightharpoonup \vect u$ in $H^2(\Omega, \R^3)$ that
    \[
    \divergence \left( \mu_n \symgrad \vect u_n \right) \rightharpoonup 
    \divergence \left( \hat \mu \symgrad \hat {\vect u} \right) \quad \text{ in } L^2(\Omega).
    \]
\end{proof}

\begin{corollary}[Anisotropic case]\label{con:lem:aniso-continuity}
    If $\vect u_n \rightharpoonup \vect u$ in $H^2(\Omega, \R^3)$ and $t_n \rightharpoonup t$ in $H^2(\Omega)$, then
    \[
    \tens C _ {t_n} \nabla \vect u_n \rightharpoonup \tens C_t \nabla \vect u \quad \text{ in } H(\operatorname{div}; \Omega, \R^3).
    \]
    for the constitutive law given in \cref{met:eqn:anisotropic_law}.
\end{corollary}

We also state the following corollary as a consequence of the trace theorem.

\begin{corollary}\label{con:lem:boundary-continuity}
    If $\vect u_n \rightharpoonup \vect u$ in $H^2(\Omega, \R^3)$ and $t_n \rightharpoonup t$ in $H^2(\Omega)$, then
    \[
    \tens C _ {t_n} \nabla \vect u_n \rightharpoonup \tens C_t \nabla \vect u \quad \text{ in } L^2(\p \Omega, \R^3).
    \]
    for either of the two systems \cref{met:eqn:isotropic_law} and \cref{met:eqn:anisotropic_law}.
\end{corollary}

\begin{lemma}\label{con:lem:joint_bound}
    Let the constitutive law be given either by \cref{met:eqn:isotropic_law} or \cref{met:eqn:anisotropic_law}.
    For $\vect u_1, \vect u_2 \in H^2(\Omega)$ and $t_1, t_2 \in H^1(\Omega)$ we have the estimate
    \[
    \norm{ \divergence \left( \tens C_{t_1}  \nabla \vect u_1 \right) - \divergence \left( \tens C_{t_2} \nabla \vect u_2 \right) }_{L^2(\Omega)}^2
    \leq C_b \left( \norm{\vect u_1 - \vect u_2}_{H^2(\Omega)}^2
    + \norm{t_1 - t_2}_{H^1(\Omega)}^2
    \right),
    \]
    where $C_b=\mathcal O(\norm{t_1}_{\mathcal C^1(\bar \Omega)} + \norm{\vect u_2}_{H^2(\Omega)})$.
\end{lemma}
\begin{proof}
    First, we consider the isotropic law:
    \begin{align*} 
&\norm{ \divergence \left( \mathbb C_{\mu_1}  \symgrad \vect u_1 \right) - \divergence \left( \mathbb C_{\mu_2} \symgrad \vect u_2 \right) }_{L^2(\Omega)}^2
\\
&\quad\quad\leq 
\norm{\divergence\left(\mathbb{C}_{\mu_1} \symgrad \left( \vect u_1 - \vect u_2 \right) \right)}_{L^2(\Omega)}^2
+ \norm{\divergence\left( \left( \mu_1 - \mu_2 \right) \symgrad \vect u_2 \right)}_{L^2(\Omega)}^2 \\
&\quad\quad\leq 4\left( \norm{\lambda}_{\mathcal C^1(\bar \Omega)} + 2\norm{\mu_1}_{\mathcal C^1(\bar \Omega)} \right)
\norm{\vect u_1 - \vect u_2}_{H^2(\Omega)}^2 + 
4\norm{\vect u_2}_{H^2(\Omega)}^2\norm{\mu_1 - \mu_2}_{H^1(\Omega)}^2.
\end{align*}
    It remains to estimate the anisotropic part:
    \begin{align*} 
&\norm{ \divergence \left( t_1 \tens R  \nabla \vect u_1 \right) - \divergence \left( t_2 \tens R \nabla \vect u_2 \right) }_{L^2(\Omega)}^2
\\
&\quad\quad\leq 
\norm{\divergence\left(t_1 \tens R \nabla \left( \vect u_1 - \vect u_2 \right) \right)}_{L^2(\Omega)}^2
+ \norm{\divergence\left( \left( t_1 - t_2 \right) \tens R \nabla \vect u_2 \right)}_{L^2(\Omega)}^2 \\
&\quad\quad\leq 2 \norm{\tens R}_{\mathcal C^1(\bar \Omega)} \norm{t_1}_{\mathcal C^1(\bar \Omega)}
\norm{\vect u_1 - \vect u_2}_{H^2(\Omega)}^2 + 
2\norm{\tens R}_{\mathcal C^1(\bar \Omega)}\norm{\vect u_2}_{H^2(\Omega)}^2\norm{t_1 - t_2}_{H^1(\Omega)}^2.
\end{align*}
\end{proof}

\section{Stability estimate for the anisotropic system}\label{sta:sec}

Following the strategy in \cite{di_fazio_stable_2017}, we derive the stability estimate for the anisotropic system. 
The main difference compared to \cite{di_fazio_stable_2017} is that, instead of the lower bound given by eq. (30) in \cite{di_fazio_stable_2017}, namely
\[
\int_{B_r\left(x_0\right)}|\hat{\nabla} u|^2 d x \geq C_d\left(\frac{r}{d}\right)^K\|g\|_{H^{1 / 2}(\partial \Omega)}^2,
\]
which establishes a polynomial rate in dependence of the radius, we will use a weaker bound leading to the Hölder exponent $\delta=1/2$ instead of $\delta=N/(N+1)>1/2$ for a possibly large $N$. 
For the sake of completeness, we outline the proof strategy from \cite{di_fazio_stable_2017} in a more abstract setting allowing for a weaker lower bound at the cost of a slightly worse exponent. 
Recall that $\vect u^* \in H^1(\Omega, \R^3)$ denotes a solution to the homogeneous system of\cref{the:eqn:ground-truth} and $\vect u \in H^1(\Omega, \R^3)$ denotes a solution to the system of \cref{eqn:second-system}.

\begin{theorem}\label{the:thm:stabilityabstract}
    Introduce $\varphi=t-t^*$ as the difference in the parameters and assume that there is some $x_0 \in \Omega_d$ such that the maximum of $\varphi$ is attained in $x_0$, i.e. $\left|\varphi\left( x_0\right)\right|=\displaystyle \max _{\bar{\Omega}_d}|\varphi|$.
    Furthermore, assume that we have the upper bound
    \[
    s\int_\Omega |\varphi| |\tens P \nabla \vect u^*|^2 dx \leq \int_\Omega |\varphi| \tens S \nabla \vect u^* \cdot \nabla \vect u^* dx \leq \varepsilon,
    \]
    where $\tens P, \tens S \in \R^{3\times 3 \times 3 \times 3}$ are prescribed fourth-order tensors,
    as well as the lower bound
    \begin{equation}\label{the:eqn:interior-estimate-2}
    \int_{B_r(x_0)} |\tens P \nabla \vect u^*|^2 dx \geq \eta > 0 \quad \text{for some } 0<r\leq d.
    \end{equation}
    Additionally, we assume that $\varphi$ is Lipschitz continuous with constant $2M>0$.
    
    Then, we have the following interior stability estimate for $\varphi$:
    \[
    \norm{\varphi}_{\mathcal C(\bar \Omega_d)} \leq C \varepsilon^{\delta},
    \]
    with a constant $C=C(M,d,\eta,s)>0$ and $\delta \in [1/2,1)$.
\end{theorem}
\begin{proof}
    Let $x_0 \in \Omega_d$ be such that $\left|\varphi\left( x_0\right)\right|=\displaystyle \max _{\bar{\Omega}_d}|\varphi|$. Since $\varphi$ is Lipschitz continuous, we get
    \begin{equation*}
        \left|\varphi\left(x_0\right)\right| \leq |\varphi(x)|+2 M r \quad \forall x \in B_r\left(x_0\right), r \in(0, d].
    \end{equation*}

    Multiplying by $|\tens P \nabla \vect u^*|^2$ and integrating over $B_r(x_0)$ yields
    \begin{equation}\label{the:eqn:interior-estimate-1}
        \left|\varphi\left(x_0\right)\right| \int_{B_r\left(x_0\right)}|\tens P \nabla \vect u^*|^2 d x
        \leq \int_{B_r\left(x_0\right)}|\varphi(x)| |\tens P \nabla \vect u^*|^2 d x +2 M r \int_{B_r\left(x_0\right)}|\tens P \nabla \vect u^*|^2 d x.
    \end{equation}
    By the assumption on the upper bound, we have 
    \begin{equation*}
        \int_{B_r\left(x_0\right)}|\varphi(x)| |\tens P \nabla \vect u^*|^2 d x \leq \varepsilon/s.
    \end{equation*}
    Dividing \cref{the:eqn:interior-estimate-1} by $\int_{B_r\left(x_0\right)}|\tens P \nabla \vect u^*|^2 d x$ and using (\ref{the:eqn:interior-estimate-2}) leads to
    \begin{equation*}
        \left|\varphi\left(x_0\right)\right| \leq \frac{\varepsilon}{s\eta} + 2Mr \quad \quad \text{ for all } 0<r\leq d,
    \end{equation*}
    or, by defining $\lambda=r/d \in (0,1]$,
    \begin{equation}\label{the:eqn:interior-estimate-3}
        \left|\varphi\left(x_0\right)\right| \leq \frac{\varepsilon}{s\eta} + 2Md\lambda
        \leq \frac{\varepsilon}{s\eta}\lambda^{-1} + 2Md\lambda\quad \quad \text{ for all } 0<\lambda \leq 1.
    \end{equation}

    Let us consider
    \begin{equation*}
        \bar{\lambda} = \left( \frac{\varepsilon}{2Md \eta} \right)^{1/2},
    \end{equation*}
    and make a case distinction.
    If $\bar{\lambda}\leq 1$, inserting it in (\ref{the:eqn:interior-estimate-3}) gives
    \begin{equation*}
        \left|\varphi\left(x_0\right)\right| \leq 
        \left(\frac{2Md}{\eta}\right)^{1/2}\left( \frac{1}{s} + 1 \right) \varepsilon^{1/2} \leq C(M,d,\eta,s) \varepsilon^{1/2}.
    \end{equation*}
    If $\bar{\lambda} > 1$, we can make use of the Lipschitz continuity of $\varphi$ to obtain $|\varphi(x_0)| \leq 2M\sup_{x,y \in \Omega_d}|x-y|$.
    Therefore, after multiplying the right-hand side with $\bar{\lambda} > 1$, we obtain
    \begin{equation*}
        \left|\varphi\left(x_0\right)\right| 
        \leq C(M,d,\eta,s) \varepsilon^{1/2}.
    \end{equation*}
\end{proof}
In the next step, we want to give the corresponding bounds for the anisotropic system.

\begin{lemma} \label{lem:integral_estimate}
    There exists a constant $C=C(\Omega, M, \norm{\vect k}_{H^{3 / 2}(\partial \Omega)}, \norm{\vect g}_{H^{3 / 2}(\partial \Omega)})>0$ not dependent on the field parameter such that
\begin{equation}\label{eqn:integral-estimate}
        \int_{\Omega} \left|\varphi \right||\nabla \vect u^*|^2 d x \leq 
        C \Big\{ \norm{\vect f}_{L^2(\Omega)} + \norm{\tens C_t \nabla \bu \cdot \bn - \vect p}_{L^2(\partial \Omega)} +
        \norm{\vect u - \vect u^*}_{L^2(\Omega)}^{1/4}
        \Big\}.
    \end{equation}
\end{lemma}
\begin{proof}
In what follows, it is more convenient to use index notation. Therefore, we implicitly assume that we sum over all indices $i,j,k,l=1,2,3$ unless otherwise stated.

We start by testing with $\mvect \zeta \in H^1(\Omega, \R^3)$ defined by
\[
\mvect \zeta=\frac{\min \left\{\max \{\varphi,0 \}, h\right\}}{h} \vect u^*,
\]
and compare the weak formulations to get
\begin{equation}\label{the:eqn:anisotropicweak}
\begin{aligned}
\int_\Omega \varphi \tens R \nabla \vect u^* : \nabla \mvect \zeta dx =
- \int_\Omega &\left[ \tens L \nabla (\vect u - \vect u^*) + t \tens R \nabla (\vect u - \vect u^*) \right] :\nabla \mvect \zeta dx \\
&+ \int_{\p \Omega}
\left( \tens C_t \nabla \vect u \cdot \vect n - \vect p \right) \cdot \mvect \zeta ds
+ \int_\Omega \vect f \cdot \mvect \zeta dx.
\end{aligned}
\end{equation}
We rewrite the left-hand side as
\begin{align*}
\int_\Omega \varphi \tens R \nabla \vect u^* : \nabla \mvect \zeta dx &=
\int_{0<\varphi<h} \varphi \tens R \nabla \vect u^* : \nabla \left( \frac{\varphi }{h} \vect u^* \right) dx
+ \int_{h< \varphi} \varphi \tens R \nabla \vect u^* : \nabla \vect u^* dx \\
&= \frac{1}{h}\int_{0<\varphi<h} \varphi \tens R \nabla \vect u^* : \left(  \vect u\otimes \nabla \varphi \right) dx 
+ \frac{1}{h}\int_{0<\varphi<h} \varphi^2 \tens R \nabla \vect u^* : \nabla \vect u^* dx \\
& \quad \quad \quad \quad + \int_{h< \varphi} \varphi \tens R \nabla \vect u^* : \nabla \vect u^* dx.
\end{align*}

Using the identity
\begin{align*}
\varphi \p_j \varphi R_{ijkl} \p_l u_k^*  u_i^* &=
\frac{1}{2} \p_j \left( \varphi^2 \right) R_{ijkl} \p_l u_k^*  u_i^* \\
&= \frac{1}{2} \p_j \left( \varphi^2 R_{ijkl} \p_l u_k^*  u_i^* \right)
- \frac{1}{2} \varphi^2\p_j \left(  R_{ijkl} \p_l u_k^*  u_i^* \right),
\end{align*}
we can further rewrite the left-hand side of \cref{the:eqn:anisotropicweak} and use the divergence theorem twice to obtain
\begin{align*}
I_1 &\coloneqq \frac{1}{h}\int_{0<\varphi<h} \varphi R \nabla \vect u^* : \left(  \vect u^* \otimes \nabla \varphi \right) dx \\
&=
\frac{1}{2h}\int_{0<\varphi<h} \p_j \left( \varphi^2 R_{ijkl} \p_l u_k^*  u_i^* \right) dx
- \frac{1}{2h}\int_{0<\varphi<h} \varphi^2\p_j \left(  R_{ijkl} \p_l u_k^*  u_i^* \right) dx \\
&= -\frac{h}{2}\int_{\varphi=h} R_{ijkl} \p_l u_k^*  u_i^* n_j ds
- \frac{1}{2h}\int_{0<\varphi<h} \varphi^2\p_j \left(  R_{ijkl} \p_l u_k^*  u_i^* \right) dx \\
&= -\frac{h}{2}\int_{\varphi>h} \p_j \left( R_{ijkl} \p_l u_k^*  u_i^* \right) dx
- \frac{1}{2h}\int_{0<\varphi<h} \varphi^2\p_j \left(  R_{ijkl} \p_l u_k^*  u_i^* \right) dx,
\end{align*}
so that we can estimate
\[
|I_1| \leq \frac{h}{2} \int_{0<\varphi} \left| \p_j \left(  R_{ijkl} \p_l u_k^*  u_i^* \right) \right| dx.
\]

In the next step, we look at
\[
|I_2| \coloneqq \left| \frac{1}{h}\int_{0<\varphi<h} \varphi^2 \tens R \nabla \vect u^* : \nabla \vect u^* dx \right|
\leq h \int_{0<\varphi<h} \left|\tens R \nabla \vect u^* : \nabla \vect u^* \right| dx.
\]
Overall, we obtain
\begin{align*}
\int_\Omega \varphi \tens R \nabla \vect u^* : \nabla \mvect \zeta dx &\geq
- |I_1| - |I_2| + \int_{h< \varphi} \varphi R \nabla \vect u^* : \nabla \vect u^* dx \\
&\geq - \frac{h}{2} \int_{0<\varphi} \left| \p_j \left(  R_{ijkl} \p_l u_k^*  u_i^* \right) \right| dx \\
& \quad \quad \quad \quad - h \int_{0<\varphi<h} \left| \tens R \nabla \vect u^* : \nabla \vect u^* \right| dx
+ \int_{h< \varphi} \varphi \tens R \nabla \vect u^* : \nabla \vect u^* dx.
\end{align*}
Now let us estimate the right-hand side of \cref{the:eqn:anisotropicweak}. By \Cref{met:ass:elasticity_boundedness}, we can estimate 
\begin{align*}
\left| \int_\Omega \left[ \tens L \nabla (\vect u - \vect u^*) + t \tens R \nabla (\vect u - \vect u^*) \right] :\nabla \mvect \zeta dx \right|
& \leq K C_{it} \Bigg( \varepsilon 
        \left( \norm{\vect k}_{H^{3 / 2}(\partial \Omega)} + \norm{\vect g}_{H^{3 / 2}(\partial \Omega)}\right) \\
        & \quad \quad \quad +\frac{C(\Omega)}{\varepsilon}\norm{\vect u - \vect u^*}_{L^2(\Omega)} \Bigg) \norm{\vect g}_{H^{3 / 2}(\partial \Omega)},
\end{align*}
where we used the same estimates as already established in \cite[Lemma 4.1]{di_fazio_stable_2017}.

For the boundary term, we get similarly
\begin{align*}
\left| \int_{\p \Omega}
\left( \tens C_t \nabla \vect u \cdot \vect n - \vect p \right) \cdot \mvect \zeta ds \right|
\leq \norm{\tens C_t \nabla \vect u \cdot \vect n - \vect p}_{L^2(\p \Omega)} \norm{\vect u ^*}_{L^2(\p \Omega)},
\end{align*}
and for the force term
\begin{align*}
\left| \int_\Omega \vect f \cdot \mvect \zeta dx \right| \leq \norm{\vect f}_{L^2(\Omega)} \norm{\vect u ^*}_{L^2(\Omega)}.
\end{align*}
Merging the individual estimates, we get
\begin{align*}
\int_{h< \varphi} \varphi \tens R \nabla \vect u^* : \nabla \vect u^* dx &\leq
\frac{h}{2} \int_{0<\varphi} \left| \p_j \left(  R_{ijkl} \p_l u_k^*  u_i^* \right) \right| dx
+ h \int_{0<\varphi<h} \left| \tens R \nabla \vect u^* : \nabla \vect u^* \right| dx \\
&+ K C_{it} \Bigg( \varepsilon 
        \left( \norm{\vect k}_{H^{3 / 2}(\partial \Omega)} + \norm{\vect g}_{H^{3 / 2}(\partial \Omega)}\right) \\
        & \quad \quad \quad +\frac{C(\Omega)}{\varepsilon}\norm{\vect u - \vect u^*}_{L^2(\Omega)} \Bigg) \norm{\vect g}_{H^{3 / 2}(\partial \Omega)} \\
&+ \norm{\tens C_t \nabla \vect u \cdot \vect n - \vect p}_{L^2(\p \Omega)} \norm{\vect u ^*}_{L^2(\p \Omega)} \\
&+ \norm{\vect f}_{L^2(\Omega)} \norm{\vect u ^*}_{L^2(\Omega)}.
\end{align*}
Now, since
\begin{align*}
\int_{0 <\varphi< h} \varphi \tens R \nabla \vect u^* : \nabla \vect u^* dx \leq
h \int_{0 <\varphi< h} \left| \tens R \nabla \vect u^* : \nabla \vect u^* \right| dx \leq
h \int_{\Omega} \left| \tens R \nabla \vect u^* : \nabla \vect u^* \right| dx,
\end{align*}
we obtain
\begin{align*}
\int_{0 <\varphi} \varphi \tens R \nabla \vect u^* : \nabla \vect u^* dx &\leq
\frac{h}{2} \int_{\Omega} \left| \p_j \left(  R_{ijkl} \p_l u_k^*  u_i^* \right) \right| dx
+ 2h \int_{\Omega} \left|\tens R \nabla \vect u^* : \nabla \vect u^* \right| dx \\
&+ K C_{it} \Bigg( \varepsilon 
        \left( \norm{\vect k}_{H^{3 / 2}(\partial \Omega)} + \norm{\vect g}_{H^{3 / 2}(\partial \Omega)}\right) \\
        & \quad \quad \quad +\frac{C(\Omega)}{\varepsilon}\norm{\vect u - \vect u^*}_{L^2(\Omega)} \Bigg) \norm{\vect g}_{H^{3 / 2}(\partial \Omega)} \\
&+ \norm{\tens C_t \nabla \vect u \cdot \vect n - \vect p}_{L^2(\p \Omega)} \norm{\vect u ^*}_{L^2(\p \Omega)} \\
&+ \norm{\vect f}_{L^2(\Omega)} \norm{\vect u ^*}_{L^2(\Omega)}.
\end{align*}
Note that the right-hand side is independent of the signs of the terms of the left-hand side. In particular, deriving the same estimate for $-\varphi$ leads to
\begin{equation}\label{the:eqn:almostaniso}
\begin{aligned}
\int_{0 <\varphi} |\varphi| \tens R \nabla \vect u^* : \nabla \vect u^* dx &\leq
\frac{h}{2} \int_{\Omega} \left| \p_j \left(  R_{ijkl} \p_l u_k^*  u_i^* \right) \right| dx
+ 2h \int_{\Omega} \left|\tens R \nabla \vect u^* : \nabla \vect u^* \right| dx \\
&+ K C_{it} \Bigg( \varepsilon 
        \left( \norm{\vect k}_{H^{3 / 2}(\partial \Omega)} + \norm{\vect g}_{H^{3 / 2}(\partial \Omega)}\right) \\
        & \quad \quad \quad +\frac{C(\Omega)}{\varepsilon}\norm{\vect u - \vect u^*}_{L^2(\Omega)} \Bigg) \norm{\vect g}_{H^{3 / 2}(\partial \Omega)} \\
&+ \norm{\tens C_t \nabla \vect u \cdot \vect n - \vect p}_{L^2(\p \Omega)} \norm{\vect u ^*}_{L^2(\p \Omega)} \\
&+ \norm{\vect f}_{L^2(\Omega)} \norm{\vect u ^*}_{L^2(\Omega)}.
\end{aligned}
\end{equation}
In order to get the absolute value on the left-hand side, we make use of the specific structure $R_{ijkl} = r_j \delta_{jl} \delta_{ik}$ and the convexity condition to estimate pointwise
\[
\theta |\nabla \vect u^*|^2 \leq \tens R \nabla \vect u^* : \nabla \vect u^* \leq M |\nabla \vect u^*|^2,
\]
which allows us to further estimate the left-hand side of \cref{the:eqn:almostaniso} from below.

Also, using the pointwise estimate
\[
\left| \p_j \left(  R_{ijkl} \p_l u_k^*  u_i^* \right) \right| \leq
M |\nabla \vect u^*|| \vect u^*|+M|\divergence \left( \vect u^* \nabla \vect u^* \right)|,
\]
we can use the already established estimates to bound the additional terms on the right-hand side of \cref{the:eqn:almostaniso} from above.

In total, choosing $\varepsilon = \norm{\vect u - \vect u^*}_{L^2(\Omega)}^{1/2}$ and $h=\norm{\vect u - \vect u^*}_{L^2(\Omega)}^{1/4}$, we get the upper bound
\begin{align*}
\int_{0 <\varphi} |\varphi| |\nabla \vect u^* | dx \leq
\theta^{-1} K(\Omega) M C_{it} &\norm{\vect g}_{H^{3 / 2}(\partial \Omega)} \Big\{ \norm{\vect f}_{L^2(\Omega)} + \norm{\tens C_t \nabla \vect u \cdot \vect n - \vect p}_{L^2(\partial \Omega)}  \Big. \\ \Big.
        &+
        \left( \norm{\vect k}_{H^{3 / 2}(\partial \Omega)} + \norm{\vect g}_{H^{3 / 2}(\partial \Omega)} + 1 \right)
        \norm{\vect u - \vect u^*}_{L^2(\Omega)}^{1/4}
        \Big\}.
\end{align*}
\end{proof}
To derive the corresponding lower bound, we want to state one more auxiliary lemma.

\begin{lemma}[Korn's inequality, {{\cite[Lemma 3.5]{alessandrini_detecting_2004}}}]
\label{the:lem:korn}
    There exists a constant $C>0$ such that for every $\vect u \in H^1(\Omega, \R^n)$ we have
    \[
    \int_{\Omega} \left|\vect u- \vect u_\Omega- \mat W_\Omega x\right|^2 dx \leq C|\Omega|^2 \int_{\Omega}|\symgrad \vect u|^2dx,
    \]
    where
    \begin{equation}\label{the:eqn:kornobjects}
    \begin{aligned}
&\vect u_\Omega=\frac{1}{\left|\Omega\right|} \int_{\Omega} \vect u \; dx,\\
&\mat W_\Omega=\frac{1}{\left|\Omega\right|} \int_{\Omega}\symgrad \vect u \; dx.
\end{aligned}
    \end{equation}
\end{lemma}

Next, we recall a unique continuation principle for the gradient, as originally derived in \cite{lin_detecting_2018}.

\begin{lemma}[{{\cite[Lemma 5.2]{lin_detecting_2018}}}]
    Let $d>0$ such that $\Omega_{d}$ is still a connected domain. Furthermore, set $r=d/9$. Then, for every $x_0 \in \Omega_{9d}$, the gradient of the ground-truth solution $\vect u^*$ fulfils 
    \[
    \int_{B_r(x_0)}\left|\nabla \vect u^* \right|^2 \geq C_d \int_{\Omega}\left|\nabla \vect u^*\right|^2,
    \]
    where $C_d = C\left(\theta, M,\Omega, d\right)$.
\end{lemma}
\begin{remark}
    The restriction $r=d/9$ arises from a technical condition on the radii of the three-spheres inequality, as derived in \cite[Corollary 4.10]{lin_detecting_2018}. Without modifying the proof of \cite[Lemma 5.2]{lin_detecting_2018}, this constraint can be readily generalised to $r=d/\epsilon$ for any $\epsilon>1$. The broader framework of three-spheres inequalities presented in \cite{alessandrini_stability_2009} illustrates how such constraints on the radii can be bypassed. Nevertheless, for our purposes, this restriction does not pose any significant issues.
\end{remark}
We also need a global bound on the gradient.
\begin{lemma}
    There exists a constant $C=C(\Omega,\vect g, \delta_0)>0$, where $\delta_0>0$ is the constant from \Cref{met:ass:boundary-data}, such that for all $\vect u \in H^1(\Omega, \R^3)$ subject to $\vect u = \vect g$ on $\partial \Omega$ we have 
    \[
    \norm{\nabla \vect u}_{L^2(\Omega)} \geq C.
    \]
\end{lemma}
\begin{proof}
    Due to Korn's inequality as given by Lemma \ref{the:lem:korn}, we have 
    \[
    \norm{\vect u - (\vect u_\Omega + \mat W_\Omega x)}_{L^2(\Omega)}^2 \leq C(\Omega)\norm{\nabla \vect u}_{L^2(\Omega)}^2,
    \]
    where $\vect u_\Omega + \mat W_\Omega$ are as defined in \cref{the:eqn:kornobjects}.
    Moreover,
    \[
    \norm{\nabla (\vect u - (\vect u_\Omega + \mat W_\Omega x))}_{L^2(\Omega)}^2 \leq \norm{\nabla \vect u}_{L^2(\Omega)}^2 + \norm{\mat W_\Omega}_{L^2(\Omega)}^2 = 3 \norm{\nabla \vect u}_{L^2(\Omega)}^2,
    \]
    so that by the trace theorem (e.g. \cite[A 8.6]{alt_linear_2016}) and \Cref{met:ass:boundary-data}, we get
    \begin{align*}
    \norm{\nabla \vect u}_{L^2(\Omega)}^2 &\geq C(\Omega) \norm{\vect u - (\vect u_\Omega + \mat W_\Omega x)}_{H^1(\Omega)}^2 \\
    & \geq C(\Omega)C_{t} \norm{\vect g - (\vect u_r + \mat W_r x)}_{H^{1/2}(\partial \Omega)}^2 \\
    & \geq C(\Omega) C_{t} \delta_0.
    \end{align*}
\end{proof}
We combine the previous two results in order to be able to apply the general Theorem \ref{the:thm:stabilityabstract}.
\begin{corollary}
    Let $d>0$ such that $\Omega_{9d}$ is still a connected domain. Furthermore, set $r=d/9$. Then, for every $x_0 \in \Omega_{d}$, the gradient of the ground-truth solution $\vect u^*$ fulfils 
    \[
    \int_{B_r(x_0)}\left|\nabla u_0\right|^2 \geq C_d,
    \]
    where $C_d = C\left(\theta, M,d, \vect g, \delta_0, \Omega\right)$.
\end{corollary}
This leads to a very similar stability estimate as for the isotropic case.
\begin{corollary}
    Let $d>0$ be such that $\Omega_d$ is still a connected domain. Then, there exists a constant $C=C(\Omega, M, \norm{\vect k}_{H^{3 / 2}(\partial \Omega)}, \norm{\vect g}_{H^{3 / 2}(\partial \Omega)})>0$ and $\delta \in [1/2,1)$ such that
    \begin{equation*}
        \left\|t-t^*\right\|_{L^{\infty}\left(\Omega_d\right)} \leq 
        C \Big\{ \norm{\vect f}_{L^2(\Omega)} + \norm{\tens {C}_t \symgrad \vect u \cdot \vect n - \vect p}_{L^2(\partial \Omega)} +
        \norm{\vect u - \vect u^*}_{L^2(\Omega)}^{1/4}
        \Big\}^\delta.
    \end{equation*}
    Moreover, the constants $C, \delta$ do not depend on the field parameters $t, t^*$.
\end{corollary}

\printbibliography

\end{document}